\def\draftdate{January 8, 2026}

\documentclass{amsart}
\usepackage{amssymb}
\usepackage{xy}
\usepackage{bm}
\usepackage[breaklinks,linktocpage=true,colorlinks=true,linkcolor=black]{hyperref}

\newcommand{\overto}[1]{\xrightarrow{\,#1\,}}
\newcommand{\overfrom}[1]{\xleftarrow{\,#1\,}}

\let\Lim\lim
\def\lim{\Lim\nolimits}

\mathchardef\varDelta="7101

\newcommand{\phat}[1][{p}]{^{\wedge}_{#1}}

\let\iso\cong
\let\sma\wedge

\renewcommand{\to}{\mathchoice{\longrightarrow}{\rightarrow}{\rightarrow}{\rightarrow}}

\newcommand{\sto}{\rightarrow}

\DeclareMathAlphabet{\catsymbfont}{U}{rsfs}{m}{n}

\newcommand{\bC}{{\mathbb{C}}}

\newcommand{\bF}{{\mathbb{F}}}

\newcommand{\bS}{{\mathbb{S}}}
\newcommand{\bQ}{{\mathbb{Q}}}

\newcommand{\bT}{{\mathbb{T}}}
\newcommand{\bZ}{{\mathbb{Z}}}

\newcommand{\Chi}{\mathrm{X}}

\def\quickop#1{\expandafter\DeclareMathOperator\csname
#1\endcsname{#1}}
\quickop{id}\quickop{Id}\quickop{holim}\quickop{hocolim}\quickop{op}
\quickop{co}\quickop{Ar}\quickop{sing}\quickop{Hom}\quickop{w}\quickop{Ho}
\quickop{ob}\quickop{diag}\quickop{Cyc}\quickop{Fib}\quickop{Cof}
\quickop{tr}\quickop{trc}\quickop{Gal}\quickop{colim}\quickop{triv}
\quickop{red}\quickop{hoeq}\quickop{inc}\quickop{cyc}

\numberwithin{equation}{section}

\newtheorem{thm}[equation]{Theorem}
\newtheorem*{thm*}{Theorem}
\newtheorem{cor}[equation]{Corollary}
\newtheorem{lem}[equation]{Lemma}
\newtheorem{prop}[equation]{Proposition}
\newtheorem{conj}[equation]{Conjecture}
\theoremstyle{definition}
\newtheorem{defn}[equation]{Definition}
\newtheorem{notn}[equation]{Notation}

\newtheorem*{ql}{Question}
\theoremstyle{remark}

\xyoption{arrow}
\xyoption{curve}
\xyoption{matrix}
\xyoption{cmtip}
\SelectTips{cm}{}
\newcommand{\term}[1]{\textit{#1}}

\newdir{ >}{{}*!/-5pt/\dir{>}}

\begin{document}

\title[Chromatic convergence]%
{Chromatic convergence for the algebraic $K$-theory of the sphere spectrum}

\author{Andrew J. Blumberg}
\address{Department of Mathematics, Columbia University, 
New York, NY \ 10027}
\email{blumberg@math.columbia.edu}
\thanks{The first author was supported in part by NSF grants
DMS-1812064, DMS-2104420, DMS-2405029}
\author{Michael A. Mandell}
\address{Department of Mathematics, Indiana University,
Bloomington, IN \ 47405}
\email{mmandell@indiana.edu}
\thanks{The second author was supported in part by NSF grants
DMS-1811820, DMS-2104348, DMS-2405030}
\author{Allen Yuan}
\address{Department of Mathematics, Northwestern University,
Evanston, IL \ 60208}
\email{allenyuan@northwestern.edu}
\thanks{The third author was supported in part by NSF grant DMS-2002029}

\date{\draftdate}

\begin{abstract}
We show that the map from $K({\mathbb S})$ to its chromatic completion
is a connective cover and identify the fiber in $K$-theoretic terms.
We combine this with recent work of Land-Mathew-Meier-Tamme to prove a
form of ``Waldhausen's Chromatic Convergence Conjecture'': we
show that the map 
$K({\mathbb S}_{(p)})_{(p)}\to \mathop{\rm holim} K(L^{f}_{n}{\mathbb S})_{(p)}$
is the inclusion of a wedge summand.
\end{abstract}

\maketitle

%%%%%%%%%%%%%%%%%%%%%%%%%%%%%%%%%%%%%%%%
\section{Introduction}

The chromatic approach to homotopy theory is geared to study spectra
that are \term{chromatically convergent}, meaning that (for a fixed 
prime $p$) the natural map 
\[
X_{(p)}\to \holim_{n}L_{n}X
\]
is a weak equivalence, where $L_{n}$ denotes localization with respect
to the Johnson-Wilson spectrum $E(n)$, or equivalently, the sum of the
Morava $K$-theory spectra $K(0)\vee \dotsb \vee K(n)$.  A celebrated
theorem of Hopkins-Ravenel, the \term{chromatic convergence theorem},
is that finite spectra are chromatically convergent. 

Waldhausen~\cite[\S4]{Waldhausen-Chromatic} (following work of
Dwyer-Friedlander~\cite{DwyerFriedlander-EtaleK},
Thomason~\cite{Thomason-QLC,ThomasonEtale}) tied the chromatic approach 
to algebraic $K$-theory in his reformulation of the
Quillen-Lichtenbaum conjecture: the Quillen-Lichtenbaum conjecture
holds for a ring or scheme $Z$ at the prime $p$ if and only if the
localization map $K(Z)\to L_{1}K(Z)$ is an isomorphism on homotopy
groups in high degrees.  Mitchell~\cite{Mitchell-MoravaKofAlgK}
(Corollary) shows that the maps $L_{n}K(Z)\to L_{1}K(Z)$ are always
weak equivalences.  We can then rephrase the Quillen-Lichtenbaum
conjecture for $Z$ as the assertion that $K(Z)$ is chromatically
semiconvergent in the following sense.

\begin{defn}
A spectrum $X$ is \term{chromatically $N$-semiconvergent} for some $N\in \bZ$
if the \term{chromatic completion map}
\[
X_{(p)}\to \holim_{n} L_{n}X
\]
has $N$-truncated fiber (homotopy groups are zero in degree above $N$).
It is \term{chromatically semiconvergent} if it is chromatically 
$N$-semiconvergent for some $N$, or equivalently, when the chromatic
completion map induces an isomorphism on homotopy groups in high degrees.
\end{defn}

Mitchell's result and elementary properties of chromatic
semiconvergence discussed in Section~\ref{sec:elementary} then imply the
following theorem.

\begin{thm}\label{main:QL}
For a ring $R$ (not necessarily commutative) or a scheme $Z$, the map $K(R)\to
L_{1}K(R)$ or $K(Z)\to L_{1}K(Z)$ is an isomorphism on homotopy groups
in high degrees if and only if $K(R)$ or $K(Z)$ is chromatically
semiconvergent.
\end{thm}

The chromatic semiconvergence perspective now allows a straightforward
generalization of the question underlying the Quillen-Lichtenbaum
conjecture: for which ring spectra $R$ is $K(R)$ chromatically
semiconvergent?    
As an example of this perspective, we offer the following theorem
(proved in Section~\ref{sec:elementary}) that %(in the cases it covers)
reduces chromatic semiconvergence of algebraic $K$-theory to the
corresponding question for $TC$.  The $TC$ question tends to be reasonably
accessible to homotopy theoretic arguments; see for example
\cite[\S4]{BMY1}.  For a more general result 
closely related to Quillen-Lichtenbaum conjectures for rings, see
Corollary~\ref{cor:QLDGM} and the remarks that follow it.

\begin{thm}\label{main:tracemethods}
Let $R$ be a connective ring spectrum with $\pi_{0}R=\bZ$,
$\bZ_{(p)}$, or $\bZ\phat$.  Then $K(R)$ is chromatically
semiconvergent if and only if $TC(R)$ is chromatically semiconvergent
if and only if $TC(R;p)\phat$ is chromatically semiconvergent.
\end{thm}

We showed in~\cite[1.3]{BMY1} that for $S=\bS,\bS_{(p)}$, or
$\bS\phat$, the spectrum $TC(S)\phat$ is chromatically
convergent (see also Proposition~\ref{prop:fixTCS} below).  This resolves the
question for $\bS$, $\bS_{(p)}$, and $\bS\phat$.

\begin{cor}\label{main:cor}
For $S=\bS,\bS_{(p)}$, or $\bS\phat$, $K(S)$ is chromatically
semiconvergent. 
\end{cor}

The main purpose of this paper is to study the fiber of the chromatic
completion map for $K(\bS)$, $K(\bS_{(p)})$, and $K(\bS\phat)$.
Throughout this paper we use the following notation:

\begin{notn}\label{notn:F}
For any spectrum $X$, let $F(X)$ denote the homotopy fiber of the
chromatic completion map $X_{(p)}\to \holim L_{n}X$.
\end{notn}

The case of
$F(K(\bS\phat))$ is easiest to describe.

\begin{thm}\label{main:FSphat}
For $K(\bS\phat)$, 
\[
F(K(\bS\phat))\simeq \Sigma^{-3}H(\bZ/p^{\infty}).
\]
In particular, for $q>0$, 
\[
\pi_{-q}\holim L_{n}K(\bS\phat)\iso \begin{cases}
\bZ/p^{\infty}&q=2\\
0&q\neq 2.
\end{cases}
\]
\end{thm}

The case of $L_{n}K(\bS)$ is more interesting.

\begin{thm}\label{main:FS}
For $K(\bS)$, 
\[
F(K(\bS))
\simeq
\Sigma^{-3}I_{\bZ/p^{\infty}} K(\bZ),
\]
where
$I_{\bZ/p^{\infty}}$ denotes the $p$-local Brown-Comenetz dual.  In
particular, for $q>0$
\[
\pi_{-q}(\holim L_{n}K(\bS))\iso 
\pi_{-q-1}(\Sigma^{-3}I_{\bZ/p^{\infty}} K(\bZ))
\iso \Hom(\pi_{q-2}K(\bZ),\bZ/p^{\infty})
\]
\end{thm}

The case of $\bS_{(p)}$ is described by a fiber sequence.

\begin{thm}\label{main:FSploc}
For $K(\bS_{(p)})$, $F(K(\bS_{(p)}))$ fits in a fiber sequence
\[
F(K(\bS_{(p)}))\to \bigvee_{\ell\neq p}\Sigma^{-2}I_{\bZ/p^{\infty}}K(\bF_{\ell})
\to\Sigma^{-2}I_{\bZ/p^{\infty}}K(\bZ)\to \Sigma \dotsb, 
\]
where the wedge is over primes $\ell\neq p$, and the map $I_{\bZ/p^{\infty}}K(\bF_{\ell})
\to I_{\bZ/p^{\infty}}K(\bZ)$ is Brown-Comenetz dual to the map
$K(\bZ)\to K(\bF_{\ell})$ induced by $\bZ\to \bF_{\ell}$. 
\end{thm}

In older drafts of this paper, the description of the map in the case
$p=2$ was conjectural, but this case has since been resolved by
Cho~\cite{Cho-ktpd}. See Conjecture~\ref{conj:ktpd} and
Theorem~\ref{thm:FSplocFinal} for a more specific formulation. 
			      
As particular consequences of the theorems above, for $S=\bS$,
$\bS_{(p)}$, or $\bS\phat$, the chromatic completion map
$K(S)_{(p)}\to \holim L_{n}K(S)$ is a connective cover. 

So far all our discussion has been for the chromatic localization
functors $L_{n}$.  However, all previous work on Quillen-Lichtenbaum
conjectures for ring spectra have been based on periodic
homotopy groups and the finite telescopic localization functors
$L^{f}_{n}$ (see, for example, \cite{Miller-FiniteLocalizations} for a
definition). In particular, Rognes' redshift
conjectures~\cite[p.~8]{Rognes-Oberwolfach2000},
\cite[4.1(a)]{Rognes-MSRI}, \cite[6.0.5]{HahnWilson-RedshiftBPn} 
and results in this direction by
Angelini-Knoll~\cite{AngeliniKnoll-Periodicity},
Angelini-Knoll-Quigley~\cite{AngeliniKnollQuigley-Segal},
Ausoni-Rognes~\cite{AusoniRognes},
Carmeli-Schlank-Yanovski~\cite{CSY-AmbiHeight},
Clausen-Mathew-Naumann-Noel~\cite{CMNN-Descent},
Hahn-Wilson~\cite{HahnWilson-RedshiftBPn},
Land-Mathew-Meier-Tamme~\cite{LMMT}, and
Veen~\cite{Veen-DetectingPeriodic} ask or answer the following
question about particular ring spectra $R$:

\begin{ql}[$K$-Theoretic Periodicity Question for $R$]
Does the map 
\[
K(R)\to L_{n}^{f}K(R)
\]
induce an isomorphism on homotopy groups in all
sufficiently high degrees?
\end{ql}

There is an essentially unique natural transformation of
localizations $L_{n}^{f}\to L_{n}$, and Ravenel's telescope
conjecture~\cite[10.5]{Ravenel-Periodic} was the assertion that this is
a weak equivalence.  This holds for $n=1$, but is false for
$n>1$~\cite{Telescope2023}.  As a consequence, the 
relationship between $K$-theory chromatic semiconvergence questions
and the cited redshift work above is not easily described.

Instead, the $K$-Theoretic Periodicity Question for ring spectra $R$
is intimately related to the concept of \emph{telescopic}
semiconvergence:
Say that a spectrum $X$ is telescopically convergent when the
telescopic completion map 
\[
X_{(p)}\to \holim_{n} L^{f}_{n}X
\]
is a weak equivalence and telescopically semiconvergent when it
induces an isomorphism on homotopy groups in high degrees.  
In light of the failure of the telescope conjecture,
Barthels~\cite[5.6]{Barthel-ChromaticConvergence} implies that
telescopic convergence and chromatic convergence are different.
However, many of
the formal properties of telescopic semiconvergence are analogous to
the properties of chromatic semiconvergence; in particular, for a ring
$R$ (as opposed to ring spectrum) or scheme $Z$, the map $K(R)\to
L_{1}^{f}K(R)$ or $K(Z)\to L^{f}_{1}K(Z)$ induces an isomorphism in
homotopy groups in high degrees if and only if $K(R)$ or $K(Z)$ is
telescopically semiconvergent.  
More generally, for a spectrum $X$,
the map $X\to L^{f}_{n}X$ induces an isomorphism on homotopy groups in
high degrees if and only if $X$ is telescopically semiconvergent and
its periodic homotopy groups $T(m)_{*}X$ vanish for all $m>n$ (where
$T(m)=v_{m}^{-1}V(m)$ for some chosen finite type $m$ spectrum
$V(m)$). Moreover, the main result of
Land-Mathew-Meier-Tamme~\cite{LMMT} implies that $T(m)_{*}K(R)$
vanishes for all $m>n$ when $T(m)_{*}R$ vanishes for all $m\geq
n$ or equivalently (see~\cite[2.3]{LMMT}) when
$K(m)_{*}R$ vanishes for all $m\geq n$.  Thus, in the context of
redshift conjectures, the $K$-theoretic Periodicity Question for $R$
is equivalent to telescopic semiconvergence for $K(R)$.  If we ask
instead the question: 

\begin{ql}
Is $K(R)$ telescopically semiconvergent?
\end{ql}

\noindent This then expands the scope of Quillen-Lichtenbaum
conjectures for ring spectra beyond those that vanish above a given
chromatic or telescopic height.  In particular, we can ask this
question for $\bS$ or $MU$.

In the setting of the current paper, the analogue of
Theorem~\ref{main:QL} (and the more general Corollary~\ref{cor:QLDGM}
in Section~\ref{sec:elementary}) hold with telescopic semiconvergence
replacing chromatic semiconvergence by the analogous proof with
$L^{f}_{n}$ replacing $L_{n}$.  We do not, however, have the
telescopic analogue of Corollary~\ref{main:cor} (telescopic
semiconvergence of $K(\bS)$): it is not known whether $TC(\bS)$ or
$TC(\bS;p)\phat$ is telescopically semiconvergent.  Indeed, it is not
known whether the sphere spectrum itself is telescopically semiconvergent.
Because $\bS$ is a wedge summand of $K(\bS)$, telescopic
semiconvergence of $K(\bS)$ would imply telescopic semiconvergence of
$\bS$ and hence telescopic semiconvergence of all finite spectra.
Such a result is currently out of reach; the best we can say is that
the map $\bS_{(p)}\to \holim L^{f}_{n}\bS$ is an inclusion of a wedge
summand.

Returning to the ideas in Waldhausen~\cite{Waldhausen-Chromatic}, the
original Ravenel conjectures~\cite[\S10]{Ravenel-Periodic} motivated
Waldhausen to make the following conjecture:

\begin{conj}[Waldhausen's Chromatic Convergence Conjecture]\label{conj:W1}
\[
K(\bS_{(p)})\overto{\simeq}\holim K(L^{f}_{n} \bS).
\]
\end{conj}

In fact, Waldhausen's original conjecture concerns the category of
finite spectra with telescopic equivalences, which
(by~\cite[1.6.7]{WaldhausenKT}) gives the ``free'' algebraic
$K$-theory of $L^{f}_{n}\bS$ (the $K$-theory of finite cell complexes
rather than of perfect complexes).  The previous conjecture is
equivalent (by~\cite[1.10.2]{ThomasonTrobaugh}) to Waldhausen's
original conjecture plus an additional conjecture that $K_{0}(\bS)\to
\lim K_{0}(L^{f}_{n} \bS)$ is an isomorphism and $\lim\nolimits^{1}
K_0(L^{f}_{n} \bS)=0$.

If telescopic semiconvergence of $\bS$ fails, it still makes sense to
ask for the map in Conjecture~\ref{conj:W1} to be the inclusion of a
wedge summand.  Combining our chromatic semiconvergence theorem for
$K(\bS_{(p)})$ with the results of
Land-Mathew-Meier-Tamme~\cite{LMMT}, we prove the following theorem in
Section~\ref{sec:wspf}.

\begin{thm}\label{main:waldsplit}
The natural map
\[
K(\bS_{(p)})_{(p)}\to \holim K(L^{f}_{n} \bS)_{(p)}
\]
is the inclusion of a wedge summand.
\end{thm}

Land-Mathew-Meier-Tamme~\cite[2.5]{LMMT} formulate a
finite localization functor $L^{p,f}_{n}$ that is a version of
$L^{f}_{n}$ which does not $p$-localize (does not invert the primes
$\ell\neq p$).  Using this localization functor, we also have the
following result for $K(\bS)$.

\begin{thm}\label{main:Swaldsplit}
The natural map
\[
K(\bS)_{(p)}\to \holim K(L^{p,f}_{n} \bS)_{(p)}
\]
is the inclusion of a wedge summand.
\end{thm}

At the time Waldhausen was writing, it was widely expected that the
telescope conjecture would hold for all $n$
and~\cite{Waldhausen-Chromatic} assumes this for some its conclusions;
for this reason, the conjecture is sometimes formulated with the
chromatic localizations.

\begin{conj}[Waldhausen's Chromatic Convergence Conjecture, 2nd Version]
\label{conj:W2}
\[
K(\bS_{(p)})\overto{\simeq}\holim K(L_{n} \bS).
\]
\end{conj}

For the chromatic localizations $L_{n}$, the analogue of
Theorem~\ref{main:waldsplit} is a slightly weaker version of
Conjecture~\ref{conj:W2} that seems approachable and still retains much
of the power of the original.  We discuss the current obstacle to a
complete proof at the end of Section~\ref{sec:wspf}.

\subsection*{Conventions}
Throughout this paper the unmodified term ``ring spectrum'' and the
term ``commutative ring spectrum'' 
mean $A_{\infty}$ and $E_{\infty}$ ring spectrum or (equivalently)
$\bS$-algebra and commutative $\bS$-algebra.

\subsection*{Acknowledgments}
The authors thank Ben Antieau, Prasit Bhattacharya, Jeremy Hahn, Mike
Hopkins, and John Rognes for helpful conversations, and Myungsin Cho
for corrections. The final draft benefited from corrections and many
helpful suggestions from an anonymous referee.

%%%%%%%%%%%%%%%%%%%%%%%%%%%%%%%%%%%%%%%%%%%%%%%%%%%%%%%%%%%%%%%%%%%%%%
\section{Basic properties of chromatic semiconvergence and proof of
Theorem~\ref{main:tracemethods}}
\label{sec:elementary}

In this section, we review some easy generalities about chromatic
convergence and chromatic semiconvergence.
Theorem~\ref{main:tracemethods} follows from these properties, the
resolved Quillen-Lichtenbaum conjecture, and the
work of Mitchell~\cite{Mitchell-MoravaKofAlgK},
Dundas-Goodwillie-McCarthy~\cite{GoodwillieHN,Dundas-RelK,McCarthy},
and Hesselholt-Madsen~\cite{HM2} on $K$-theory of rings and ring
spectra.

Because localizations and homotopy limits commute with fiber
sequences, cofiber sequences, and suspensions, so does chromatic
convergence.  This leads to the following observation about chromatic
semiconvergence. (See Notation~\ref{notn:F} for $F$.)

\begin{prop}\label{prop1}
If $W\to X\to Y\to \Sigma \dotsb$ is a fiber sequence of spectra then
so is 
\[
\holim L_{n}W\to \holim L_{n}X\to \holim L_{n}Y\to \Sigma \dotsb
\]
and so is
\[
F(W)\to F(X)\to F(Y)\to \Sigma \dotsb.
\]
In particular, if $W,Y$ are chromatically $N$-semiconvergent, then so
is $X$. 
\end{prop}

We used the following proposition in the introduction in the example
of algebraic $K$-theory spectra, and we record it here for
convenience.  

\begin{prop}\label{prop2}
Let $X$ be a spectrum and assume that there exists $m\geq 0$ such that
$K(n)_{*}X=0$ for all $n>m$.  Then the natural map $\holim L_{n}X\to
L_{m}X$ is a weak equivalence.
\end{prop}

Because $N$-truncated spectra are $K(n)$-acyclic for all $n>0$, the
previous propositions then imply the following one.

\begin{prop}\label{prop3}
Let $X$ be a spectrum and assume that there exists $m\geq 0$ and $N\in
\bZ$ such that the localization map $X_{(p)}\to L_{m}X$ is $N$-truncated.
Then $F(X)$ is weakly equivalent to the
homotopy fiber of $X_{(p)}\to L_{m}X$ and $X$ is chromatically
$N$-semiconvergent.  In particular, if $X$ is
$L_{m}$-local, then $X$ is chromatically convergent.
\end{prop}

The cofiber of the map from the localization to the completion
$X_{(p)}\to X\phat$ is always rational (and hence $L_{0}$-local).
Combining Propositions~\ref{prop1} and~\ref{prop2}, we get the
following result. 

\begin{prop}\label{prop4}
For any spectrum $X$ the square 
\[
\xymatrix{%
X_{(p)}\ar[r]\ar[d]&\holim L_{n}X\ar[d]\\
X\phat\ar[r]&\holim L_{n}(X\phat)
}
\]
is homotopy cartesian. In particular, the natural map
\[
F(X)\to F(X\phat)
\]
is a weak equivalence and 
a spectrum $X$ is
chromatically $N$-semiconvergent if and only if 
$X\phat$ is chromatically $N$-semiconvergent.
\end{prop}

We now apply the previous observations to $K$-theory.  For any
connective ring spectrum $R$, Celebrated work of
Dundas-Goodwillie-McCarthy~\cite{Dundas-RelK, GoodwillieHN, McCarthy}
(see also~\cite[VII.2.2.1]{DundasGoodwillieMcCarthy}) gives a homotopy
cartesian square
\[
\xymatrix{%
K(R)\ar[r]\ar[d]&K(\pi_{0}R)\ar[d]\\TC(R)\ar[r]&TC(\pi_{0}R)
}
\]
where $TC$ denotes Goodwillie's integral $TC$~\cite{Goodwillie-MSRI}
(see also~\cite[VI.3.3.1]{DundasGoodwillieMcCarthy}), the horizontal
maps are the linearization maps (induced by the map of ring
spectra $R\to H\pi_{0}R$) and the vertical maps are the cyclotomic
trace maps. Applying chromatic
completion $\holim L_{n}(-)$, we get a homotopy cartesian square
\[
\xymatrix{%
\holim L_{n}K(R)\ar[r]\ar[d]&\holim L_{n}K(\pi_{0}R)\ar[d]\\\holim
L_{n}TC(R)\ar[r]&\holim L_{n}TC(\pi_{0}R).
}
\]
Because $K(\pi_{0}R)$ and $TC(\pi_{0}R)$ are $K(\bZ)$-modules and
$K(n)_{*}K(\bZ)=0$ for $n>1$ by the work of
Mitchell~\cite[Thm~A]{Mitchell-MoravaKofAlgK}, we have that 
\[
K(n)_{*}K(\pi_{0}R)=0 \qquad \text{and}\qquad K(n)_{*}TC(\pi_{0}R)=0
\]
for $n>1$.  By Proposition~\ref{prop2}, we have that the maps
\[
\holim L_{n}K(\pi_{0}R)\to L_{1}K(\pi_{0}R)
\qquad \text{and}\qquad
\holim L_{n}TC(\pi_{0}R)\to L_{1}TC(\pi_{0}R)
\]
are weak equivalences.  This proves the following theorem.

\begin{thm}\label{thm:L1DGM}
Let $R$ be a connective ring spectrum.  Then the diagram
induced by the trace and linearization maps
\[
\xymatrix{%
\holim L_{n}K(R)\ar[r]\ar[d]&L_{1}K(\pi_{0}R)\ar[d]\\\holim
L_{n}TC(R)\ar[r]&L_{1}TC(\pi_{0}R)
}
\]
is homotopy cartesian.
\end{thm}

Combining with the elementary properties of chromatic semiconvergence
in the propositions above, we get the following corollary.  In it
$TC(R;p)$ denotes $p$-typical $TC$ and we are using the well-known
result that the canonical map $TC(R)\phat\to TC(R;p)\phat$ is a weak
equivalence~\cite[14.1.(ii)]{Goodwillie-MSRI}, \cite[2.2]{BMY1}. 

\begin{cor}\label{cor:QLDGM}
Let $R$ be a connective ring spectrum and assume that $K(\pi_{0}R)$
and $TC(\pi_{0}R)$ are chromatically semiconvergent.  Then $K(R)$ is
chromatically semiconvergent if and only if $TC(R)$ is chromatically
semiconvergent if and only if $TC(R;p)\phat$ is chromatically semiconvergent.
\end{cor}

We note that the requirement that $K(\pi_{0}R)$ is chromatically
semiconvergent is equivalent to requiring that the localization map 
\[
K(\pi_{0}R)\phat\to (L_{1}K(\pi_{0}R))\phat \simeq L_{K(1)}K(\pi_{0}R)
\]
be $N$-truncated for some $N$; this is precisely the requirement that
the Quillen-Lichtenbaum conjecture hold for $\pi_{0}R$ (as
reformulated by Waldhausen).  Work of Hesselholt-Madsen~\cite[Thm~D,
Add~6.2]{HM2} shows that under finite (module) generation hypotheses, the
connective cover of $TC(\pi_{0}R)\phat$ is equivalent to
$K((\pi_{0}R)\phat)\phat$.  In such cases, the requirement that 
$TC(\pi_{0}R)$ is chromatically semiconvergent is also a
Quillen-Lichtenbaum conjecture requirement.

Theorem~\ref{main:tracemethods} is an immediate consequence of the
previous corollary since the Quillen-Lichtenbaum conjecture holds for
$\bZ$, $\bZ_{(p)}$, and $\bZ\phat$, and therefore the algebraic
$K$-theory and $TC$ of these rings are chromatically semiconvergent.

Corollary~\ref{main:cor} follows from Theorem~\ref{main:tracemethods}
and the following proposition, which is essentially~\cite[1.3]{BMY1}.

\begin{prop}\label{prop:fixTCS}
Let $S=\bS$, $\bS_{(p)}$, or $\bS\phat$.  Then $TC(S;p)\phat$ is
chromatically convergent.
\end{prop}

\begin{proof}
In all three cases, we have by~\cite{BHM} that $TC(S;p)\phat$ is
$\bS\phat\vee X\phat$ where $X$ is the fiber of the $\bT$-transfer
$\Sigma \Sigma^{\infty}_{+}B\bT\to \bS$.  We have that $\bS$ is
chromatically convergent by the Hopkins-Ravenel chromatic convergence
theorem, and $\Sigma^{\infty}_{+}B \bT$ is chromatically convergent
by~\cite[4.2]{BMY1}.
\end{proof}

The proof of the Hopkins-Ravenel chromatic convergence theorem (and
Barthel's elaboration~\cite[3.8]{Barthel-ChromaticConvergence} used
via~\cite{BMY1} in the argument above) proves
more than a weak equivalence; it actually proves pro-isomorphisms on
each homotopy group.  Specifically, for
fixed $q\in \bZ$, the map of towers in $n$
\[
\{\pi_{q}(TC(\bS;p)\phat)\}\to \{ \pi_{q}(L_{n}(TC(\bS;p)\phat)\}
\]
is a pro-isomorphism, where on the left we have the constant tower.  In
particular, when $S=\bS$, $\bS_{(p)}$, or $\bS\phat$, then for each
fixed $q$, the towers $\{\pi_{q}(L_{n}(TC(S;p)\phat))\}$ and 
therefore $\{\pi_{q}(L_{n}TC(S))\}$ are pro-constant.  Since for $Z=\pi_{0}S$,
the towers $\{\pi_{q}(L_{n}K(Z))\}$ and $\{\pi_{q}(L_{n}TC(Z))\}$ are
actually constant (for $n\geq 1$), we get the following proposition.  

\begin{prop}\label{prop:proconst}
Let $S=\bS$, $\bS_{(p)}$, or $\bS\phat$, and fix $q\in \bZ$.  The
tower in $n$
\[
\{\pi_{q}(L_{n}K(S))\}
\]
is pro-constant.
\end{prop}

%%%%%%%%%%%%%%%%%%%%%%%%%%%%%%%%%%%%%%%%%%%%%%%%%%%%%%%%%%%%%%%%%%%%%%%%
\section{Tools for studying fiber of the chromatic completion map on 
\texorpdfstring{$K(\bS)$}{K(S)}}
\label{sec:tools}

In this section, we start work on Theorems~\ref{main:FS},
\ref{main:FSploc}, and~\ref{main:FSphat}, which identify the fiber $F$
of
the chromatic completion map for $K(\bS)$, $K(\bS_{(p)})$, and
$K(\bS\phat)$, respectively.  In the first part of this section, we
argue that the statements reduce to assertions for the
$p$-completions.  In the case of $K(\bS\phat)$, this is all that is
needed to prove Theorem~\ref{main:FSphat}, and the proof is completed
in this section.  Theorems~\ref{main:FS} and~\ref{main:FSploc} for
$\bS$ and $\bS_{(p)}$ are stated in terms of the Brown-Comenetz
duality, and for the $p$-complete versions of these statements, we
work with Anderson duality, which is an equivalent theory but involves
a suspension.  The second part of the section introduces terminology
for Anderson duality that makes it easier to state and prove precise
duality statements for spectra and for maps.

In this section, we treat the three cases as uniformly as possible,
writing $S$ for $\bS$, $\bS_{(p)}$, or $\bS\phat$ and $Z$ for
$\pi_{0}S$.  Rewriting in terms of our notation $F$ for the homotopy
fiber of the chromatic completion map (Notation~\ref{notn:F}), Proposition~\ref{prop:fixTCS} states that
$F(TC(S;p)\phat)$ is trivial.  Proposition~\ref{prop4} then implies
that $F(TC(S))$ is trivial.  Using this, 
Theorem~\ref{thm:L1DGM} gives a fiber sequence
\[
F(K(S))\to F(K(Z))\to F(TC(Z))\to \Sigma \dotsb.
\]
Writing $\trc_{Z}\colon K(Z)_{(p)}\to TC(Z)_{(p)}$ for the
$p$-localization of the cyclotomic trace and
$\Fib(\trc_{Z})$ for its fiber, swapping homotopy limits, we get a
fiber sequence  
\begin{equation}\label{eq:MainFib}
F(K(S)) \to \Fib(\trc_{Z})\to L_{1}\Fib(\trc_{Z})\to \Sigma \dotsb.
\end{equation}
The map $\Fib(\trc_{Z})\to L_{1}\Fib(\trc_{Z})$ is the chromatic
completion map,
so~\eqref{eq:MainFib} gives a weak equivalence
\[
F(K(S))\simeq F(\Fib(\trc_{Z})).
\]
Since the $L_{1}$-localization map is always a rational
equivalence, we get the following immediate consequence.

\begin{prop}\label{prop:integral}
For $S=\bS$, $\bS_{(p)}$, or $\bS\phat$, the fiber $F(K(S))$ of the
chromatic completion map $K(S)_{(p)}\to \holim L_{n}K(S)$ is
rationally trivial.  In particular, 
\[
F(K(S))\simeq (F(K(S))\phat)\sma \Sigma^{-1}M_{\bZ/p^{\infty}},
\]
where $M_{\bZ/p^{\infty}}$ denotes the Moore spectrum for
$\bZ/p^{\infty}$
\end{prop}

\begin{proof}
Smashing any spectrum $X$ with the fiber sequence
\[
\Sigma^{-1}M_{\bZ/p^{\infty}}\to \bS_{(p)} \to \bS_{\bQ}\to M_{\bZ/p^{\infty}}\to \Sigma \dotsb
\]
shows that when $X$ is $p$-local and rationally trivial, 
\[
X\simeq X\sma \Sigma^{-1}M_{\bZ/p^{\infty}}.
\]
The fiber of the $p$-completion map $X\to X\phat$ is (by definition)
$\bS/p$-acyclic, and therefore $\bS/p^{n}$-acyclic for all $n$, and so
$M_{\bZ/p^{\infty}}$-acyclic, since $M_{\bZ/p^{\infty}}\simeq \hocolim
\bS/p^{n}$.  Smashing the fiber sequence for the $p$-completion map with
$\Sigma^{-1}M_{\bZ/p^{\infty}}$ then gives a weak equivalence
\[
X\sma \Sigma^{-1}M_{\bZ/p^{\infty}}\overto{\simeq} X\phat \sma \Sigma^{-1}M_{\bZ/p^{\infty}}.
\qedhere
\]
\end{proof}

Theorem~\ref{main:FSphat} (the case $S=\bS\phat$) is now immediate:
work of Hesselholt-Madsen~\cite[Thm.~B,D]{HM2} implies 
\[
\Fib(\trc_{\bZ\phat})\phat\simeq \Sigma^{-2} H\bZ\phat.
\]
The $p$-completion of $L_{1}\Fib(\trc_{\bZ\phat})$ is therefore
trivial, and we see that $F(K(\bS\phat))\phat\simeq
\Sigma^{-2}H\bZ\phat$.  By the proposition, we get 
\[
F(K(\bS\phat))\simeq \Sigma^{-2}H\bZ\phat\sma
\Sigma^{-1}M_{\bZ/p^{\infty}}\simeq \Sigma^{-3}H\bZ/p^{\infty}.
\]
This completes the proof of Theorem~\ref{main:FSphat}.

Both remaining cases, $S=\bS$ and $S=\bS_{(p)}$, require the
identification of $p$-local Brown-Comenetz duals of certain spectra and of
certain maps.  For a spectrum $X$, the $p$-local Brown-Comenetz dual
$I_{\bZ/p^{\infty}}X$ is characterized (up to isomorphism in the
stable category) by its representable functor 
\[
[-,I_{\bZ/p^{\infty}}X]\iso \Hom(\pi_{0}((-)\sma X),\bZ/p^{\infty}).
\]
For $p$-complete spectra, it works slightly better to work in terms of
$\bZ\phat$-Anderson duals.  We use the following terminology:

\begin{defn}
For $p$-complete spectra $X,Y$, an
\term{Anderson duality pairing} on $X,Y$ is a homomorphism
\[
\mu\colon \pi_{0}(X\sma Y;\bZ/p^{\infty})\to \bZ/p^{\infty}
\]
such that the induced maps
\begin{gather*}
\pi_{q}(X;\bZ/p^{\infty})\otimes \pi_{-q}(Y)\to
\pi_{0}(X\sma Y;\bZ/p^{\infty})\overto{\mu} 
\bZ/p^{\infty}\\
\pi_{q}(X)\otimes \pi_{-q}(Y;\bZ/p^{\infty})\to
\pi_{0}(X\sma Y;\bZ/p^{\infty})\overto{\mu} 
\bZ/p^{\infty}
\end{gather*}
are perfect pairings for all $q\in \bZ$, i.e., if they are adjoint to
isomorphisms  
\begin{gather*}
\pi_{q}(X;\bZ/p^{\infty})\to \Hom(\pi_{-q}(Y),
\bZ/p^{\infty})\\
\pi_{-q}(Y;\bZ/p^{\infty})\to \Hom(\pi_{q}(X),
\bZ/p^{\infty}).
\end{gather*}
\end{defn}

The homomorphism $\mu$ induces a map of spectra $X\sma Y\sma
M_{\bZ/p^{\infty}}$ to $I_{\bZ/p^{\infty}}\bS$, which is then adjoint
to maps 
\[
X\sma M_{\bZ/p^{\infty}}\to I_{\bZ/p^{\infty}}Y
\qquad \text{and}\qquad 
Y\sma M_{\bZ/p^{\infty}}\to I_{\bZ/p^{\infty}}X.
\]
The perfect pairing condition implies that these maps are weak
equivalences.  The following proposition then restates this in the
form we use to study~\eqref{eq:MainFib}. 

\begin{prop}\label{prop:BCd}
For $p$-complete spectra $X,Y$, an Anderson duality pairing
induces isomorphisms in the stable category
\[
X\overto{\simeq}\Sigma^{-1}(I_{\bZ/p^{\infty}}Y)\phat
\qquad \text{and}\qquad 
Y\overto{\simeq}\Sigma^{-1}(I_{\bZ/p^{\infty}}X)\phat
\]
and isomorphisms in the stable category
\[
X\sma \Sigma^{-1}M_{\bZ/p^{\infty}}\overto{\simeq}\Sigma^{-1}I_{\bZ/p^{\infty}}Y
\qquad \text{and}\qquad 
Y\sma \Sigma^{-1}M_{\bZ/p^{\infty}}\overto{\simeq}\Sigma^{-1}I_{\bZ/p^{\infty}}X.
\]
\end{prop}

As an example, we have an
Anderson duality pairing for $KU\phat$ with itself
given by the map
\[
\pi_{0}(KU\phat \sma KU\phat;\bZ/p^{\infty})\to 
\pi_{0}(KU\phat;\bZ/p^{\infty})\iso \bZ/p^{\infty}
\]
with the first map induced by multiplication and the second map the
isomorphism 
\[
\pi_{0}(KU\phat;\bZ/p^{\infty})
\iso \pi_{0}(KU\phat)\otimes \bZ/p^{\infty}
\iso \bZ/p^{\infty}
\]
deriving from the canonical isomorphism 
$\pi_{0}(KU)\iso \bZ$ (induced by inclusion of the unit).  For $p>2$,
$KU\phat$ decomposes into a wedge
of suspensions of the so-called Adams summand $L\phat$, 
\[
KU\phat\simeq L\phat \vee \Sigma^{2}L\phat \vee \dotsb \vee \Sigma^{2p-4}L\phat.
\]
The Adams summand admits an Anderson duality pairing with itself via
the multiplication $L\phat \sma L\phat\to L\phat$ and canonical
isomorphism $\pi_{0}(L\phat;\bZ/p^{\infty})\iso \bZ/p^{\infty}$ as
well.  For $p=2$, we have a Anderson duality pairing for $KO\phat[2]$
and $\Sigma^{4}KO\phat[2]$ again coming from multiplication and
a choice of isomorphism
$\pi_{0}(\Sigma^{4}KO\phat[2];\bZ/2^{\infty})\iso \bZ/2^{\infty}$; see
\cite[4.16]{Anderson-UCT} for a proof.

For the theorems of the introduction, we also need to be able to
identify when maps are Brown-Comenetz dual.  We use the following terminology.

\begin{defn}\label{defn:BCdcrit}
Given Anderson duality
pairings $(X_{1},Y_{1},\mu_{1})$ and $(X_{2},Y_{2},\mu_{2})$, we say
that maps $f\colon X_{1}\to X_{2}$ and $g\colon Y_{2}\to Y_{1}$ are
\term{Anderson dual} when the diagram
\[
\xymatrix{%
\pi_{0}(X_{1}\sma Y_{2};\bZ/p^{\infty})\ar[r]^{g_{*}}\ar[d]_{f_{*}}
&\pi_{0}(X_{1}\sma Y_{1};\bZ/p^{\infty})\ar[d]^{\mu_{1}}\\
\pi_{0}(X_{2}\sma Y_{2};\bZ/p^{\infty})\ar[r]_{\mu_{2}}
&\bZ/p^{\infty}
}
\]
commutes.
\end{defn}

Unwinding the adjunctions, we get the following proposition.

\begin{prop}
Let $(X_{1},Y_{1},\mu_{1})$ and $(X_{2},Y_{2},\mu_{2})$ be Anderson
duality pairings of $p$-complete spectra.  Given $f\colon X_{1}\to X_{2}$ and $g\colon
Y_{2}\to Y_{1}$, the following are equivalent:
\begin{enumerate}
\item The maps $f$ and $g$ are Anderson dual.
\item The map $g$ coincides with
$\Sigma^{-1}(I_{\bZ/p^{\infty}}f)\phat$ 
under the
induced weak equivalences of Proposition~\ref{prop:BCd}.
\item The map  $f$
coincides with $\Sigma^{-1}(I_{\bZ/p^{\infty}}g)\phat$ 
under the
induced weak equivalences of Proposition~\ref{prop:BCd}.
\end{enumerate}
\end{prop}

\begin{proof}
If $f$ and $g$ are Anderson dual, then by the characterization of maps
into $I_{\bZ/p^{\infty}}$ in the stable category as a representable
functor, we have that the diagram (in the stable category) on the left
\[
\xymatrix{%
(X_{1}\sma Y_{2})\sma M_{\bZ/p^{\infty}}\ar[r]^{g_{*}}\ar[d]_{f_{*}}
&(X_{1}\sma Y_{1})\sma M_{\bZ/p^{\infty}}\ar[d]^{I\mu_{1}}
&X_{1}\sma Y_{2}\ar[r]^{g_{*}}\ar[d]_{f_{*}}
&X_{1}\sma Y_{1}\ar[d]^{I'\mu_{1}}\\
(X_{2}\sma Y_{2})\sma M_{\bZ/p^{\infty}}\ar[r]_-{I\mu_{2}}
&I_{\bZ/p^{\infty}}
&X_{2}\sma Y_{2}\ar[r]_-{I'\mu_{2}}
&\Sigma^{-1}(I_{\bZ/p^{\infty}})\phat
}
\]
commutes, where $I\mu_{i}$ is the map representing $\mu_{i}$.  Using
the adjunction of the functors $(-)\sma M_{\bZ/p^{\infty}}$ and
$\Sigma^{-1}(-)\phat$, we deduce that the diagram on the right commutes, where
the arrows $I'\mu_{i}$ are adjoint to the arrows $I\mu_{i}$.  Using
the adjunction of smash product $(-)\sma Y_{i}$ and function spectra
$F(Y_{i},-)$ functors, together with the canonical isomorphism 
\[
F(Y_{i},\Sigma^{-1}(I_{\bZ/p^{\infty}})\phat)\iso
\Sigma^{-1}(I_{\bZ/p^{\infty}}Y_{i})\phat,
\]
it then follows that the diagram
\[
\xymatrix{%
X_{1}\ar[r]\ar[d]_{f}
&\Sigma^{-1}(I_{\bZ/p^{\infty}}Y_{1})\phat
   \ar[d]^{\Sigma^{-1}(I_{\bZ/p^{\infty}}g)\phat}\\
X_{2}\ar[r]
&\Sigma^{-1}(I_{\bZ/p^{\infty}}Y_{2})\phat
}
\]
commutes where the top horizontal arrow is adjoint to $I'\mu_{1}$ and
the bottom horizontal arrow is adjoint to $I'\mu_{2}$.  This shows (i)
implies (iii).  Running the adjunctions backwards proves (iii) implies
(i).  Symmetry of $f$ and $g$ gives (i) implies (ii) and (ii) implies (i).
\end{proof}

%%%%%%%%%%%%%%%%%%%%%%%%%%%%%%%%%%%%%%%%%%%%%%%%%%%%%%%%%%%%%%%%%%%%%%%%
\section{Proof of Theorem~\ref{main:FS} for 
\texorpdfstring{$p>2$}{p>2}}
\label{sec:Sg2}

In this section, we study the fiber of the chromatic completion map
for $K(\bS)$ in the case $p>2$.  As we explain, the results in this
section depend on recent work of the first and second author on
arithmetic duality in algebraic $K$-theory~\cite{BM-ktpd}, and the
results for $p=2$ in the next section depend on the older work of
Rognes~\cite{Rognes2} that partially inspired it.

We start from~\eqref{eq:MainFib} and Proposition~\ref{prop:integral}
which identify $F(K(\bS))$ in terms of the fiber of the cyclotomic
trace $\trc_{\bZ}$ as
\begin{equation}\label{eq:pFib}
F(K(\bS))\simeq F(\Fib(\trc_{\bZ}))\phat\sma \Sigma^{-1}M_{\bZ/p^{\infty}}.
%\simeq F(\Fib(\trc_{\bZ})\phat)\phat\sma \Sigma^{-1}M_{\bZ/p^{\infty}}
\end{equation}
In the case $p>2$, the first two authors identify
$\Fib(\trc_{\bZ})\phat$ in~\cite{BM-ktpd} in terms of the Anderson
dual of $\Sigma L_{K(1)}K(\bZ)$ by constructing
an explicit Anderson duality pairing on $L_{K(1)}K(\bZ)$, $\Sigma
L_{K(1)}\Fib(\trc_{\bZ})$.  This pairing comes from the module
multiplication
\[
L_{K(1)}K(\bZ)\sma \Sigma L_{K(1)}\Fib(\trc_{\bZ})\to \Sigma L_{K(1)}\Fib(\trc_{\bZ})
\]
and a canonical isomorphism
\[
u_{\bZ}\colon \pi_{-1}(L_{K(1)}\Fib(\trc_{\bZ});\bZ/p^{\infty})\overto{\iso}\bZ/p^{\infty}
\]
constructed in~\cite[\S1]{BM-ktpd} (compare [ibid.,(1.6)] and the
discussion following Conjecture~\ref{conj:ktpd} below). 
The $L_{K(1)}K(\bZ)$-module structure on $L_{K(1)}\Fib(\trc_{\bZ})$
comes from a $K(\bZ)\phat$-module structure on $\Fib(\trc_{\bZ})\phat$
and we have a $K(\bZ)\phat$-module structure on the fiber of the map 
\[
\Fib(\trc_{\bZ})\phat\to L_{K(1)}\Fib(\trc_{\bZ}),
\]
which is weakly equivalent to $F(\Fib(\trc_{\bZ}))\phat$.  This gives a pairing
\[
K(\bZ)\phat \sma \Sigma^{2}F(\Fib(\trc_{\bZ}))\phat\to \Sigma^{2}F(\Fib(\trc_{\bZ}))\phat.
\]
Because $\pi_{-1}(\Fib(\trc_{\bZ})\phat;\bZ/p^{\infty})=0$ and the map
$\Fib(\trc_{\bZ})\phat\to L_{K(1)}\Fib(\trc_{\bZ})$ is injective on
$\pi_{*}(-;\bZ/p^{\infty})$ (see specifics on
$\Fib(\trc_{\bZ})\phat$ and $L_{K(1)}\Fib(\trc_{\bZ})$ below), the map 
\[
\pi_{-1}(L_{K(1)}\Fib(\trc_{\bZ});\bZ/p^{\infty})\to
\pi_{-2}(F(\Fib(\trc_{\bZ}))\phat;\bZ/p^{\infty})
\]
in the long exact sequence of homotopy groups is an isomorphism.  We
prove the following theorem; together with the weak equivalence
$F(K(\bS))\simeq F(\Fib(\trc_{\bZ}))$ of~\eqref{eq:MainFib}, 
it implies Theorem~\ref{main:FS} in the case $p>2$.

\begin{thm}\label{thm:ADFS}
For $p>2$, the homomorphism
\begin{multline*}
\pi_{0}(K(\bZ)\phat \sma \Sigma^{2}F(\Fib(\trc_{\bZ}))\phat;\bZ/p^{\infty})
\to \pi_{0}(\Sigma^{2}F(\Fib(\trc_{\bZ}))\phat;\bZ/p^{\infty})\\
\iso
\pi_{-1}(L_{K(1)}\Fib(\trc_{\bZ});\bZ/p^{\infty})\overto{u_{\bZ}}\bZ/p^{\infty}
\end{multline*}
makes $K(\bZ)\phat$, $\Sigma^{2}F(\Fib(\trc_{\bZ}))\phat$ an Anderson duality pair.
In particular, it gives a weak equivalence
\[
F(\Fib(\trc_{\bZ}))\overto{\simeq}\Sigma^{-3}I_{\bZ/p^{\infty}}(K(\bZ)\phat)\simeq 
\Sigma^{-3}I_{\bZ/p^{\infty}}(K(\bZ)).
\]
\end{thm}

To prove Theorem~\ref{thm:ADFS}, we need to break $K(\bZ)\phat$ and
$F(\Fib(\trc_{\bZ}))\phat$ into ``eigenspectra'' summands.  The
spectra $L_{K(1)}K(\bZ)$ and $\Sigma L_{K(1)}\Fib(\trc_{\bZ})$ split
into the summands that would be expected from an action of
($p$-adically interpolated) Adams operations, and \cite{BM-kstub}
showed that the Anderson duality on this pair respects this
splitting. Following the notation of \cite[\S2]{BM-kstub},  
\begin{gather*}
L_{K(1)}K(\bZ)\simeq J\phat \vee Y_{0}\vee \dotsb \vee Y_{p-2}\\
L_{K(1)}\Fib(\trc_{\bZ})\simeq J\phat \vee X_{0}\vee \dotsb \vee X_{p-2}
\end{gather*}
(numbered mod $p-1$) with $Y_{i}$, $\Sigma X_{1-i}$ Anderson dual
and $J\phat$, $\Sigma J\phat$ Anderson dual under the restriction of
the pairing above by~\cite[4.1]{BM-kstub}.  Here $Y_{i}$ is the fiber of a map of the form
\[
\bigvee \Sigma^{2i-1}L\phat\to
\bigvee \Sigma^{2i-1}L\phat
\]
for some wedges of copies of $\Sigma^{2i-1}L\phat$, where as above $L\phat$ denotes
the Adams summand of $KU\phat$; it follows that $X_{k}$ is the fiber
of a map of the form
\[
\bigvee \Sigma^{2k-1}L\phat\to
\bigvee \Sigma^{2k-1}L\phat.
\]
In terms of Brown-Comenetz duality, we have:
\begin{gather*}
J\phat\simeq \Sigma^{-2}(I_{\bZ/p^{\infty}}(J\phat))\phat\\
X_{k}\simeq \Sigma^{-2}(I_{\bZ/p^{\infty}}Y_{1-k})\phat.
\end{gather*}

Again following the notation of~\cite[\S2]{BM-kstub}, we let 
\[
y_{i}=\tau_{\geq 2}Y_{i}, \qquad 
x_{k}=\tau_{\geq 2}X_{k}, \ k\neq 0,\qquad 
x_{0}=\tau_{\geq -2}X_{0}
\]
and then
\begin{gather*}
K(\bZ)\phat \simeq j\phat \vee y_{0}\vee\dotsb\vee y_{p-2}\\
\Fib(\trc_{\bZ})\phat\simeq j\phat \vee x_{0} \vee\dotsb\vee x_{p-2}.
\end{gather*}
We note that the homotopy groups of $X_{i}$ and $Y_{i}$ are
concentrated in degrees congruent to $2i-1$ and $2i-2$ mod $2(p-1)$,
so there is a lot of leeway in most of the truncations. 

The notation above gives an identification of $F(\Fib(\trc_{\bZ}))$ in
terms of Whitehead truncations
\[
\Sigma F(\Fib(\trc_{\bZ}))\phat\simeq \tau_{\leq -1}J\phat \vee
\tau_{\leq -3}X_{0} \vee \tau_{\leq 1}X_{1}\vee \dotsb 
\vee \tau_{\leq 1} X_{p-2}.
\]
We can now prove Theorem~\ref{thm:ADFS}.

\begin{proof}[Proof of Theorem~\ref{thm:ADFS}]
From the work above we see that $K(\bZ)\phat\to L_{K(1)}K(\bZ)$
induces a split injection on $\pi_{*}$ and 
$\Sigma L_{K(1)}\Fib(\trc_{\bZ})\to \Sigma^{2}
F(\Fib(\trc_{\bZ}))\phat$ induces a split surjection on $\pi_{*}$.  By
construction of the pairings, the diagrams 
\[
\xymatrix@R-.5pc{%
\pi_{q}(\Sigma L_{K(1)}\Fib(\trc_{\bZ});\bZ/p^{\infty})\ar[r]^-{\iso}\ar[d]
&\Hom(\pi_{-q}(L_{K(1)}K(\bZ)),\bZ/p^{\infty})\ar[d]\\
\pi_{q}(\Sigma^{2}F(\Fib(\trc_{\bZ}))\phat;\bZ/p^{\infty})\ar[r]
&\Hom(\pi_{-q}(K(\bZ)\phat),\bZ/p^{\infty})\\
\pi_{q}(K(\bZ)\phat;\bZ/p^{\infty})\ar[r]\ar[d]
&\Hom(\pi_{-q}(\Sigma^{2}F(\Fib(\trc_{\bZ}))\phat),\bZ/p^{\infty})\ar[d]\\
\pi_{q}(L_{K(1)}K(\bZ);\bZ/p^{\infty})\ar[r]_-{\iso}
&\Hom(\pi_{-q}(\Sigma L_{K(1)}\Fib(\trc_{\bZ})),\bZ/p^{\infty})\\
}
\]
commute, so we just need to check that the summands for $K(\bZ)\phat$
and $\Sigma^{2}F(\Fib(\trc_{\bZ}))\phat$ correspond.  For the summands 
\[
j\phat=\tau_{\geq 0}J\phat \text{ in }K(\bZ)\phat
\qquad\text{and}\qquad 
\Sigma \tau_{\leq -1}J\phat\simeq \tau_{\leq 0}\Sigma J\phat\text{ in }
\Sigma^{2}F(\Fib(\trc))\phat,
\]
the homotopy groups match up since $\pi_{-1}(J\phat)$ is torsion free.
The summands
\[
y_{0}=\tau_{\geq 2}Y_{0} \text{ in }K(\bZ)\phat
\qquad\text{and}\qquad 
\Sigma \tau_{\leq 1}X_{1}\simeq \tau_{\leq 2}\Sigma X_{1}\text{ in }
\Sigma^{2}F(\Fib(\trc))\phat,
\]
are both the trivial spectrum; see~\cite[2.4,2.9]{BM-kstub}.  Because
$X_{k}$ has homotopy groups
only in degrees congruent to $2k-1$ and $2k-2$ mod $2(p-1)$, we have
$\tau_{\leq 1}X_{k}\overto{\simeq}\tau_{\leq -3}X_{k}$ for 
$k=2,\dotsc,p-2$.   Then for all $k$, $i=p-k$ (mod $p-1$), we see that
the homotopy groups match up in the summands 
\[
y_{i}=\tau_{\geq 2}Y_{i} \text{ in }K(\bZ)\phat
\qquad\text{and}\qquad 
\Sigma \tau_{\leq -3}X_{k}\simeq \tau_{\leq -2}\Sigma X_{k}\text{ in }
\Sigma^{2}F(\Fib(\trc))\phat,
\]
since $\pi_{-3}(X_{k})=0$ for all $k$.
(We have $\pi_{-3}(L_{K(1)}\Fib(\trc_{\bZ}))=0$ since
$\pi_{1}(L_{K(1)}K(\bZ))$ is torsion free
and $\pi_{2}(L_{K(1)}K(\bZ))$ is torsion.)
\iffalse% Other reasons $\pi_{-3}X_{k}=0$
(Because the positive degree homotopy
groups of the $Y_{i}$ are torsion in even degree and free in odd
degree, we have that the homotopy groups of $X_{k}$ are zero in all odd
negative degrees.) 
(In fact, the summands $Y_{2}$ and $X_{p-2}$ are trivial.)
\fi
\end{proof}

%%%%%%%%%%%%%%%%%%%%%%%%%%%%%%%%%%%%%%%%%%%%%%%%%%%%%%%%%%%%%%%%%%%%%%%%
\section{Proof of Theorem~\ref{main:FS} for \texorpdfstring{$p=2$}{p=2}}
\label{sec:Se2}

As in the previous section, we use~\eqref{eq:MainFib} and
Proposition~\ref{prop:integral} to identify $F(K(\bS))$ in terms of
the fiber of the cyclotomic trace $\trc_{\bZ}$ as
\begin{equation}\label{eq:2Fib}
F(K(\bS))\simeq F(\Fib(\trc_{\bZ}))\phat[2]\sma \Sigma^{-1}M_{\bZ/2^{\infty}}.
%\simeq F(\Fib(\trc_{\bZ})\phat[2])\phat[2]\sma \Sigma^{-1}M_{\bZ/2^{\infty}}
\end{equation}
In the case $p=2$, work of Rognes
identifies the $2$-completion of the fiber of the cyclotomic trace as
in terms of a fibration sequence as follows.

\begin{prop}[Rognes~{\cite[3.13]{Rognes2}}]\label{prop:xi}
Let 
\[
\xi\colon \Sigma^{-2}ku\phat[2]\to \Sigma^{4}ko\phat[2]
\]
be the map characterized by the commutative diagram
\[
\xymatrix{%
\Sigma^{-2}ku\phat[2]\ar[rrrr]^-{\xi}\ar[d]&&&& \Sigma^{4}ko\phat[2]\ar[d]\\
\Sigma^{-2}KU\phat[2]\ar[r]^-{\simeq}\ar@/_1.75pc/[rrrr]_{\Xi}
&KU\phat[2]\ar[r]^-{\psi^{3}-1}
&KU\phat[2]\ar[r]^-{\simeq}
&\Sigma^{4}KU\phat[2]\ar[r]^-{\Sigma^{4}r}
&\Sigma^{4}KO\phat[2]
}
\]
where the map $r$ is the usual ``realification'' map from $KU$ to
$KO$.  Then the $2$-completion of the fiber of
the cyclotomic trace $\Fib(\trc_{\bZ})\phat[2]$ fits in a fiber
sequence
\[
\Fib(\trc_{\bZ})\phat[2]\to
\Sigma^{-2}ku\phat[2]\overto{\xi}\sigma^{4}ko\phat[2]\to \Sigma \dotsb.
\]
\end{prop}

In the statement, the vertical maps are the $K(1)$-localization maps
and the unlabeled weak equivalences are induced by complex Bott
periodicity.  The map $\psi^{3}$ is the Adams operation, which exists as a
stable map for $KU\phat[2]$.  As indicated in the diagram above, we write 
\[
\Xi\colon \Sigma^{-2}KU\phat[2]\to \Sigma^{4}KO\phat[2]
\]
for the composite on the bottom.

We have a similar fibration sequence description for $K(\bZ)\phat[2]$
as follows.

\begin{prop}\label{prop:chi}
Let
\[
\chi\colon ko\phat[2]\to \Sigma^{4}ku\phat[2]
\]
be the map characterized by the commutative diagram
\[
\xymatrix{%
ko\phat[2]\ar[rrr]^-{\chi}\ar[d]&&& \Sigma^{4}ku\phat[2]\ar[d]\\
KO\phat[2]\ar[r]^-{c}\ar@/_1.75pc/[rrr]_{\Chi}
&KU\phat[2]\ar[r]^-{\psi^{1/3}-1}
&KU\phat[2]\ar[r]^-{\simeq}
&\Sigma^{4}KU\phat[2]
}
\]
where $\psi^{1/3}$ is the inverse of $\psi^{3}$ and $c$ is the
``complexification'' map from $KO$ to $KU$.  Then $K(\bZ)\phat[2]$ fits
into a fibration sequence 
\[
K(\bZ)\phat[2]\to ko\phat[2]\overto{\chi}\Sigma^{4}ku\phat[2]\to \Sigma \dotsb.
\]
\end{prop}

As indicated in the diagram, we write 
\[
\Chi\colon KO\phat[2]\to \Sigma^{4}KU\phat[2]
\]
for the bottom composite. To see that $\chi$ is characterized by
$\Chi$, we note that
$H^{1}(ko\phat[2];\bZ\phat[2])=H^{3}(ko\phat[2];\bZ\phat[2])=0$. 

\begin{proof}
Following B\"okstedt~\cite[1.7]{Bokstedt-HatcherWaldhausen}, we
define $JK(\bZ)\phat[2]$ as the homotopy fiber of the map 
\[
\chi'\colon ko\phat[2]\to \Sigma^{4}ku\phat[2]
\]
characterized by the commutative diagram
\[
\xymatrix{%
ko\phat[2]\ar[rrr]^-{\chi'}\ar[d]&&& \Sigma^{4}ku\phat[2]\ar[d]\\
KO\phat[2]\ar[r]^-{c}\ar@/_1.75pc/[rrr]_{\Chi'}
&KU\phat[2]\ar[r]^-{\psi^{3}-1}
&KU\phat[2]\ar[r]^-{\simeq}
&\Sigma^{4}KU\phat[2]
}
\]
using $\psi^{3}$ in place of $\psi^{1/3}$.
B\"okstedt~\cite[2.1]{Bokstedt-HatcherWaldhausen} constructs a map
$\Phi \colon K(\bZ)\phat[2]\to JK(\bZ)\phat[2]$ and shows that it is a
split surjection on homotopy groups [ibid., Theorem~2].
Rognes~\cite[p.~295]{Rognes2} shows that $\Phi$ is a map of spectra
and observes that the (now affirmed) Quillen-Lichtenbaum conjecture
implies that $\Phi$ is a weak equivalence. (By
\cite[4.2]{DwyerFriedlander-ConjecturalCalculations}, the
Quillen-Lichtenbaum conjecture asserts that $K(\bZ)\phat[2]$ and
$JK(\bZ)\phat[2]$ have isomorphic homotopy groups, and these homotopy
groups are finitely generated $\bZ\phat[2]$-modules, so a split
surjection is an isomorphism.)  This gives a fibration sequence
\[
K(\bZ)\phat[2]\to ko\phat[2]\overto{\chi'}\Sigma^{4}ku\phat[2]\to \Sigma \dotsb.
\]
To compare to the fibration sequence for $\chi$, we note that the composite 
\[
\Sigma^{4}KU\phat[2]\overto{\simeq} KU\phat[2]\overto{\psi^{1/3}}
KU\phat[2]\overto{\simeq}\Sigma^{4}KU\phat[2]
\]
is $\tfrac19\Sigma^{4}\psi^{1/3}$ and we get a self-map
$\tfrac19\Sigma^{4}\psi^{1/3}$ of $\Sigma^{4}ku\phat[2]$ as the unique map
making the evident diagram with the localization map $\Sigma^{4}ku\phat[2]\to
\Sigma^{4}KU\phat[2]$ commute.  Then
\begin{gather*}
\Chi=-\tfrac19\Sigma^{4}\psi^{1/3}\circ \chi'\colon KO\phat[2]\to \Sigma^{4}KU\phat[2],\\
\chi=-\tfrac19\Sigma^{4}\psi^{1/3}\circ \chi'\colon ko\phat[2]\to \Sigma^{4}ku\phat[2].
\end{gather*}
Because $-\tfrac19\Sigma^{4}\psi^{1/3}$ is a weak equivalence, we get
the fibration sequence in the statement for 
$K(\bZ)\phat[2]$ with $\chi$.
\end{proof}

The maps $\Xi$ and $\Chi$ are related by Anderson duality
as follows.  As discussed in Section~\ref{sec:tools}, multiplication 
\[
KO\phat[2]\sma \Sigma^{4}KO\phat[2]\to \Sigma^{4}KO\phat[2]
\]
and a choice of isomorphism
$\pi_{0}(\Sigma^{4}KO\phat[2];\bZ/2^{\infty})\iso \bZ/2^{\infty}$
gives an Anderson duality pairing on $KO\phat[2],
\Sigma^{4}KO\phat[2]$.    To choose the isomorphism, we use the
isomorphism from universal coefficient theorem 
\[
\pi_{0}(\Sigma^{4}KO\phat[2];\bZ/2^{\infty})\iso 
\pi_{0}(\Sigma^{4}KO\phat[2])\otimes \bZ/2^{\infty}
\]
(since $\pi_{-1}(\Sigma^{4}KO\phat[2])=0$) and the identification of
\[
\pi_{0}(\Sigma^{4}KO\phat[2])\iso \pi_{-4}(KO\phat[2])
\]
as $\bZ\phat[2]$ with generator the element that maps to $2u^{-2}\in
\pi_{-4}KU\phat[2]$, where $u$ is the Bott element.

We also have an Anderson duality pairing
on $\Sigma^{4}KU\phat[2],\Sigma^{-2}KU\phat[2]$ induced by the
multiplication 
\[
\Sigma^{4}KU\phat[2]\sma \Sigma^{-2}KU\phat[2]\to \Sigma^{2}KU\phat[2]
\]
and the isomorphism 
\[
\pi_{0}(\Sigma^{2}KU\phat[2];\bZ/2^{\infty})\iso
\pi_{0}(\Sigma^{2}KU\phat[2])\otimes\bZ/2^{\infty}
\iso \bZ/2^{\infty}
\]
using as generator for $\pi_{0}(\Sigma^{2}KU\phat[2])\iso\bZ\phat[2]$
the image of the unit under the 
Bott periodicity isomorphism $\pi_{0}(KU\phat[2])\iso
\pi_{0}(\Sigma^{2}KU\phat[2])$, i.e., the element corresponding under
the suspension isomorphism 
\[
\pi_{0}(\Sigma^{2}KU\phat[2])\iso \pi_{-2}(KU\phat[2])
\]  
to the inverse Bott element $u^{-1}$ in $\pi_{-2}(KU\phat[2])$.

\begin{thm}\label{thm:dualXiChi}
Under the Anderson duality pairings on $KO\phat[2],\,\Sigma^{4}KO\phat[2]$ and\break
$\Sigma^{-2}KU\phat[2],\,\Sigma^{4}KU\phat[2]$ above, the maps $\Xi$ and $\Chi$
are Anderson dual.
\end{thm}

\begin{proof}
To check Definition~\ref{defn:BCdcrit}, we need to see that the
following diagram commutes.
\[
\xymatrix@C-1.25pc{%
\pi_{0}(KO\phat[2]\sma \Sigma^{-2}KU\phat[2];\bZ/2^{\infty})
  \ar[r]^-{\Xi_{*}}
  \ar[d]_-{\Chi_{*}}
&\pi_{0}(KO\phat[2]\sma \Sigma^{4}KO\phat[2];\bZ/2^{\infty})\ar[r]
&\pi_{0}(\Sigma^{4}KO\phat[2];\bZ/2^{\infty})\ar[d]^{\iso}\\
\pi_{0}(\Sigma^{4}KU\phat[2]\sma \Sigma^{-2}KU\phat[2];\bZ/2^{\infty})\ar[r]
&\pi_{0}(\Sigma^{-2}KU\phat[2];\bZ/2^{\infty})\ar[r]_-{\iso}
&\bZ/2^{\infty}
}
\]
Unwinding the definition of $\Xi$, the top composite is the induced
map on $\pi_{0}(-;\bZ/2^{\infty})$ of the map
\begin{multline*}
KO\phat[2]\sma \Sigma^{-2}KU\phat[2]
\overto{\simeq} 
KO\phat[2]\sma KU\phat[2]
\overto{\id\sma(\psi^{3}-1)}
KO\phat[2]\sma KU\phat[2]\\
\overto{\simeq}
KO\phat[2]\sma \Sigma^{4}KU\phat[2]
\overto{\id\sma \Sigma^{4}r}
KO\phat[2]\sma \Sigma^{4}KO\phat[2]
\to
\Sigma^{4}KO\phat[2].
\end{multline*}
Because $\Sigma^{4}r$ and the Bott map $KU\phat[2]\simeq
\Sigma^{4}KU\phat[2]$ are both $KO\phat[2]$-module maps, we can identify the
previous composite as the composite
\begin{multline*}
KO\phat[2]\sma \Sigma^{-2}KU\phat[2]
\overto{\simeq} 
KO\phat[2]\sma KU\phat[2]
\overto{\id\sma(\psi^{3}-1)}
KO\phat[2]\sma KU\phat[2]\\\to
KU\phat[2]\overto{\simeq}
\Sigma^{4}KU\phat[2]\overto{\Sigma^{4}r}
\Sigma^{4}KO\phat[2].
\end{multline*}
Since the $KO\phat[2]$-module structure on $\Sigma^{4}KU\phat[2]$ comes from
the map of ring spectra $c\colon KO\phat[2]\to KU\phat[2]$, we can further
identify this map as
\begin{multline*}\label{eq:Xicomp}\tag{*}
KO\phat[2]\sma \Sigma^{-2}KU\phat[2]
\overto{c\sma \id}
KU\phat[2]\sma \Sigma^{-2}KU\phat[2]
\overto{\simeq} 
KU\phat[2]\sma KU\phat[2]\\
\overto{\id\sma(\psi^{3}-1)}
KU\phat[2]\sma KU\phat[2]\to
KU\phat[2]\overto{\simeq}
\Sigma^{4}KU\phat[2]\overto{\Sigma^{4}r}
\Sigma^{4}KO\phat[2].
\end{multline*}
Going around the diagram the other way, the map 
\begin{multline*}
\pi_{0}(KO\phat[2]\sma \Sigma^{-2}KU\phat[2];\bZ/2^{\infty})
\overto{\Chi_{*}}
\pi_{0}(\Sigma^{4}KU\phat[2]\sma \Sigma^{-2}KU\phat[2];\bZ/2^{\infty})\\
\to 
\pi_{0}(\Sigma^{2}KU\phat[2];\bZ/2^{\infty})
\end{multline*}
is the induced map on $\pi_{0}(-;\bZ/2^{\infty})$ of
\begin{multline*}
KO\phat[2]\sma \Sigma^{-2}KU\phat[2]
\overto{c\sma \id}
KU\phat[2] \sma \Sigma^{-2}KU\phat[2]
\overto{(\psi^{1/3}-1)\sma \id}
KU\phat[2] \sma \Sigma^{-2}KU\phat[2]\\
\overto{\simeq}
\Sigma^{4}KU\phat[2] \sma \Sigma^{-2}KU\phat[2]
\to
\Sigma^{2}KU\phat[2]
\end{multline*}
Since the complex Bott periodicity maps are $KU\phat[2]$-module maps,
we can identify this map as the composite
\begin{multline*}\label{eq:Chicomp}\tag{**}
KO\phat[2]\sma \Sigma^{-2}KU\phat[2]
\overto{c\sma \id}
KU\phat[2]\sma \Sigma^{-2}KU\phat[2]
\overto{\simeq}
KU\phat[2] \sma KU\phat[2]\\
\overto{(\psi^{1/3}-1)\sma \id}
KU\phat[2] \sma KU\phat[2]
\to
KU\phat[2]
\overto{\simeq}
\Sigma^{2}KU\phat[2].
\end{multline*}
We have set up the duality pairings so that the diagram
\[
\xymatrix{%
\pi_{0}(KU\phat[2];\bZ/2^{\infty})\ar[r]^{\iso}\ar[d]_{\iso}
&\pi_{0}(\Sigma^{4}KU\phat[2];\bZ/2^{\infty})\ar[r]^{\Sigma^{4}r_{*}}
&\pi_{0}(\Sigma^{4}KO\phat[2];\bZ/2^{\infty})\ar[d]^{\iso}\\
\pi_{0}(\Sigma^{2}KU\phat[2];\bZ/2^{\infty})\ar[rr]_{\iso}
&&\bZ/2^{\infty}
}
\]
commutes: since $\pi_{-1}$ is zero for $KU\phat[2]$,
$\Sigma^{-2}KU\phat[2]$, and $\Sigma^{4}KO\phat$, it suffices to check
the corresponding diagram of isomorphisms
\[
\xymatrix{%
\pi_{0}(KU\phat[2])\ar[r]^{\iso}\ar[d]_{\iso}
&\pi_{0}(\Sigma^{4}KU\phat[2])\ar[r]^{\Sigma^{4}r_{*}}
&\pi_{0}(\Sigma^{4}KO\phat[2])\ar[d]^{\iso}\\
\pi_{0}(\Sigma^{2}KU\phat[2])\ar[rr]_{\iso}
&&\bZ\phat[2],
}
\]
which commutes by the choice of generators above (in  the case of
$KO\phat[2]$, this is because $r\circ c=2$).
Comparing~\eqref{eq:Xicomp} and~\eqref{eq:Chicomp}, it suffices to see
that the maps  
\begin{gather*}
KU\phat[2]\sma KU\phat[2]
\overto{\id\sma\psi^{3}-\id}
KU\phat[2]\sma KU\phat[2]\to
KU\phat[2]\\
KU\phat[2] \sma KU\phat[2]
\overto{\psi^{1/3}\sma \id-\id}
KU\phat[2] \sma KU\phat[2]
\to
KU\phat[2]
\end{gather*}
induce the same map on $\pi_{0}(-;\bZ/2^{\infty})$. This is clear
because the diagram
\[
\xymatrix@R-1pc{%
&KU\phat[2]\sma KU\phat[2]\ar[r]\ar[dd]^-{\psi^{1/3}\sma \psi^{1/3}}
&KU\phat[2]\ar[dd]^-{\psi^{1/3}}
\\
KU\phat[2]\sma KU\phat[2]
  \ar[ur]^-{\id\sma\psi^{3}}
  \ar[dr]_-{\psi^{1/3}\sma \id}\\
&KU\phat[2]\sma KU\phat[2]\ar[r]&KU\phat[2]
}
\]
commutes and $\psi^{1/3}\colon KU\phat[2]\to KU\phat[2]$ induces the
identity map on $\pi_{0}(-;\bZ/2^{\infty})$.
\end{proof}

As a consequence, we have that 
\[
\Fib(\Xi)\simeq \Sigma^{-1}(I_{\bZ/2^{\infty}}(\Cof(\Chi)))\phat[2]
\simeq \Sigma^{-2}I_{\bZ/2^{\infty}}(\Fib(\Chi))\phat[2].
\]
We get the following ``$K$-theoretic Tate-Poitou
duality'' style result at the prime $2$ as in~\cite{BM-ktpd} at odd
primes.

\begin{cor}\label{cor:ktpd}
At the prime $2$, 
\[
L_{K(1)}\Fib(\trc_{\bZ})\simeq \Sigma^{-2}(I_{\bZ/2^{\infty}}(K(\bZ)))\phat[2] \simeq
\Sigma^{-2}I_{\bZ/2^{\infty}}(L_{K(1)}K(\bZ)\sma \Sigma^{-1}M_{\bZ/2^{\infty}}).
\]
\end{cor}

The odd case in~\cite{BM-ktpd} proves a little more, identifying the
weak equivalence intrinsically as an Anderson duality pairing.
For use in Section~\ref{sec:Sploc}, we
conjecture that the corresponding refinement also holds at the
prime~2.  Since writing the original draft of this paper, the
following conjecture has been resolved in the affirmative by
Cho~\cite[1.2]{Cho-ktpd}.

\begin{conj}\label{conj:ktpd}
At the prime $2$, for an isomorphism
\[
u_{\bZ}\colon\pi_{-1}(L_{K(1)}\Fib(\trc_{\bZ});\bZ/p^{\infty})\iso\bZ/2^{\infty}
\]
the map
\begin{multline*}
\pi_{0}(L_{K(1)}(K(\bZ))\sma \Sigma L_{K(1)}(\Fib(\trc_{\bZ}));\bZ/2^{\infty})\to 
\pi_{0}(\Sigma L_{K(1)}(\Fib(\trc_{\bZ}));\bZ/2^{\infty})\\\iso \pi_{-1}(L_{K(1)}(\Fib(\trc_{\bZ}));\bZ/2^{\infty})\overto{u_{\bZ}}\bZ/2^{\infty}
\end{multline*}
induced by the $K(\bZ)$-module structure on the fiber makes
$L_{K(1)}(K(\bZ))$ and\break $\Sigma L_{K(1)}(\Fib(\trc_{\bZ}))$ into an Anderson
duality pair.
\end{conj}

We note that the conjecture does not depend on the choice of the isomorphism; the
choice in~\cite[\S1]{BM-ktpd} was specified in terms of \'etale
cohomology, but it admits a homotopy theoretic description that
generalizes to the case $p=2$ as follows.
Let $U_{1}$ be the subgroup of $2$-adic units that are congruent to
$1$ mod $4$; then $U_{1}$ is non-canonically isomorphic to
$\bZ\phat[2]$.  The $K(1)$-localization of the $U_{1}$ Moore spectrum,
$L_{K(1)}M_{U_{1}}$, is non-canonically isomorphic in the stable
category to $J\phat[2]=L_{K(1)}\bS$ and we have a canonical
identification
\[
\pi_{-1}(L_{K(1)}M_{U_{1}})\iso
H^{1}_{\Gal}((\bZ\phat[2])^{\times};U_{1})
\iso \Hom_{c}((\bZ\phat[2])^{\times},U_{1})\iso \bZ\phat[2]
\]
where $\Hom_{c}$ denotes continuous homomorphisms and the last
isomorphism uses the usual map 
\[
(\bZ\phat[2])^{\times}\iso \{\pm1\}\times U_{1}\to U_{1}
\]
as the generator of the homomorphism group.
Since $\pi_{-2}=0$, we then get a canonical isomorphism 
\[
\pi_{-1}(L_{K(1)}M_{U_{1}};\bZ/2^{\infty})\iso
\bZ/2^{\infty}.
\]
Using the usual identification $\pi_{1}(K(\bZ\phat[2]))\iso
(\bZ\phat[2])^{\times}$ and weak equivalences
\[
L_{K(1)}TC(\bZ)\phat[2]\overto{\simeq}
L_{K(1)}TC(\bZ\phat[2])\overfrom{\simeq}
L_{K(1)}K(\bZ\phat[2]),
\]
the inclusion of $U_{1}$ in $(\bZ\phat[2])^{\times}$ induces a map 
$\Sigma L_{K(1)}M_{U_{1}}\to L_{K(1)}TC(\bZ)$.
The composite map
\begin{multline*}
\bZ/2^{\infty}\iso \pi_{0}(\Sigma L_{K(1)}M_{U_{1}};\bZ/2^{\infty})\to \pi_{0}(L_{K(1)}TC(\bZ);\bZ/2^{\infty})\\
\to \pi_{-1}(L_{K(1)}\Fib(\trc_{\bZ});\bZ/2^{\infty})
\end{multline*}
is an isomorphism and the inverse gives an isomorphism $u_{\bZ}$
for Conjecture~\ref{conj:ktpd} analogous to the one used in the case
$p>2$ in the previous section (see~\cite[1.5]{BM-ktpd}). 

To calculate $F(K(\bS))\phat[2]$ by~\eqref{eq:2Fib}, we use the
fiber sequences for $\Fib(\trc_{\bZ})\phat[2]$ and 
$L_{K(1)}\Fib(\trc_{\bZ})$ in terms of topological $K$-theories.
Specifically, we have the following calculation:
\begin{equation}\label{eq:FibFxi}
\begin{aligned}
F(K(\bS))\phat[2]\simeq F(\Fib(\trc_{\bZ}))\phat[2] &\simeq
F(\Fib\bigl(\xi\colon \Sigma^{-2}ku\phat[2] \rightarrow \Sigma^{4}ko\phat[2]\bigr))\phat[2]\\
&\simeq \Fib\bigl(F(\xi)\phat[2]\colon F(\Sigma^{-2}ku\phat[2])\phat[2]\rightarrow 
F(\Sigma^{4}ko\phat[2])\phat[2]\bigr).
\end{aligned}
\end{equation}

We note that $\Sigma F(\Sigma^{-2}ku\phat[2])\phat[2]$ is the cofiber of the map
$\Sigma^{-2}ku\phat[2]\to \Sigma^{-2}KU\phat[2]$, which is by
definition the truncation
$\tau_{\leq -4}\Sigma^{-2}KU\phat[2]$.  Working in the derived
category of $ku\phat[2]$-modules, we get a commutative diagram
\[
\xymatrix{%
\Sigma^{-2}KU\phat[2]\sma \Sigma^{4}ku\phat[2]\ar[r]\ar[d]
&\Sigma^{2}KU\phat[2] \ar[d]\\
\tau_{\leq -4}\Sigma^{-2}KU\phat[2]\sma \Sigma^{4}ku\phat[2]\ar[r]&
\tau_{\leq 0}\Sigma^{-2}KU\phat[2]
}
\]
and the canonical isomorphism 
\[
\pi_{0}(\tau_{\leq 0}\Sigma^{-2}KU\phat[2];\bZ/2^{\infty})\iso
\pi_{0}(\Sigma^{-2}KU\phat[2];\bZ/2^{\infty})\iso
\bZ/2^{\infty}
\]
then gives us an Anderson duality pairing on $\tau_{\leq
-4}\Sigma^{-2}KU\phat[2]$, $\Sigma^{4}ku\phat[2]$.  Similarly, $\Sigma
F(\Sigma^{4}ko\phat[2])$ is $\tau_{\leq 0}\Sigma^{4}KO\phat[2]$, and
using the commutative diagram
\[
\xymatrix{%
\Sigma^{4}KO\phat[2]\sma ko\phat[2]\ar[r] \ar[d]
&\Sigma^{4}KO\phat[2] \ar[d] \\
\tau_{\leq 0}\Sigma^{4}KO\phat[2] \sma ko\phat[2] \ar[r]
&\tau_{\leq 0}\Sigma^{4}KO\phat[2]
}
\]
and the canonical isomorphism 
\[
\pi_{0}(\tau_{\leq 0}\Sigma^{4}KO\phat[2];\bZ/2^{\infty})\iso
\pi_{0}(\Sigma^{4}KO\phat[2];\bZ/2^{\infty})\iso
\bZ/2^{\infty}
\]
we get a map 
\[
\mu \colon \pi_{0}(\tau_{\leq 0}\Sigma^{4}KO\phat[2]\sma ko\phat[2];\bZ/2^{\infty})
\to \bZ/2^{\infty}.
\]
We claim that $\mu$ gives an Anderson duality pairing on $\tau_{\leq
0}\Sigma^{4}KO\phat[2]$, $ko\phat[2]$.  To see this, we use the commutative
diagram{%\tiny
\[
\xymatrix@-1pc{%
\pi_{-q}(\tau_{\leq 0}\Sigma^{4}KO\phat[2];A)\otimes 
\pi_{q}(ko\phat[2];B)\ar[d]
&\pi_{-q}(\Sigma^{4}KO\phat[2];A)\otimes 
\pi_{q}(ko\phat[2];B)\ar[l]_-{\iso}\ar[r]^-{\iso}\ar[d]
&\pi_{-q}(\Sigma^{4}KO\phat[2];A)\otimes 
\pi_{q}(KO\phat[2];B)\ar[d]
\\
\pi_{0}(\tau_{\leq 0}\Sigma^{4}KO\phat[2]\sma ko\phat[2];\bZ/2^{\infty})\ar[d]
&\pi_{0}(\Sigma^{4}KO\phat[2]\sma ko\phat[2];\bZ/2^{\infty})\ar[r]\ar[l]\ar[d]
&\pi_{0}(\Sigma^{4}KO\phat[2]\sma KO\phat[2];\bZ/2^{\infty})\ar[d]
\\
\pi_{0}(\tau_{\leq 0}\Sigma^{4}KO\phat[2];\bZ/2^{\infty})
&\pi_{0}(\Sigma^{4}KO\phat[2];\bZ/2^{\infty})\ar[r]_-{=}\ar[l]^-{\iso}
&\pi_{0}(\Sigma^{4}KO\phat[2];\bZ/2^{\infty})
}
\]}%
for $q\geq 0$ and either $A=\bZ$, $B=\bZ/2^{\infty}$ or
$A=\bZ/2^{\infty}$, $B=\bZ$.

\begin{thm}
Under the Anderson duality pairings above, the map 
\[
\Sigma F(\xi)\phat[2]\colon 
\Sigma F(\Sigma^{-2}ku\phat[2])\phat[2]\to
\Sigma F(\Sigma^{4}ko\phat[2])\phat[2]
\]
for the fiber in~\eqref{eq:FibFxi} is Anderson dual to the map
$\chi\colon ko\phat[2]\to \Sigma^{4}ku\phat[2]$.
\end{thm}

\begin{proof}
We need to check that the diagram
\begin{equation*}\label{eq:twoneed}\tag{*}
\xymatrix@C-.25pc{%
\pi_{0}(ko\phat[2]\sma \tau_{\leq -4}\Sigma^{-2}KU\phat[2];\bZ/2^{\infty})
  \ar[r]^-{\Sigma F(\xi)_{*}}
  \ar[d]_-{\chi_{*}}
&\pi_{0}(ko\phat[2]\sma \tau_{\leq 0}\Sigma^{4}KO\phat[2];\bZ/2^{\infty})\ar[r]
&\pi_{0}(\tau_{\leq 0}\Sigma^{4}KO\phat[2];\bZ/2^{\infty})\ar[d]^{\iso}\\
\pi_{0}(\Sigma^{4}ku\phat[2]\sma \tau_{\leq -4}\Sigma^{-2}KU\phat[2];\bZ/2^{\infty})\ar[r]
&\pi_{0}(\tau_{\leq 0}\Sigma^{-2}KU\phat[2];\bZ/2^{\infty})\ar[r]_-{\iso}
&\bZ/2^{\infty}
}
\end{equation*}
commutes.  Consider the not-necessarily commuting diagram
\begin{equation*}\label{eq:twomid}\tag{**}
\xymatrix@C-1.25pc{%
\pi_{0}(ko\phat[2]\sma \Sigma^{-2}KU\phat[2];\bZ/2^{\infty})
  \ar[r]^-{\Xi_{*}}
  \ar[d]_-{\chi_{*}}
&\pi_{0}(ko\phat[2]\sma \Sigma^{4}KO\phat[2];\bZ/2^{\infty})\ar[r]
&\pi_{0}(\Sigma^{4}KO\phat[2];\bZ/2^{\infty})\ar[d]^{\iso}\\
\pi_{0}(\Sigma^{4}ku\phat[2]\sma \Sigma^{-2}KU\phat[2];\bZ/2^{\infty})\ar[r]
&\pi_{0}(\Sigma^{-2}KU\phat[2];\bZ/2^{\infty})\ar[r]_-{\iso}
&\bZ/2^{\infty}
}
\end{equation*}
and the diagram
\begin{equation*}\label{eq:twodone}\tag{***}
\xymatrix@C-1.25pc{%
\pi_{0}(KO\phat[2]\sma \Sigma^{-2}KU\phat[2];\bZ/2^{\infty})
  \ar[r]^-{\Xi_{*}}
  \ar[d]_-{\Chi_{*}}
&\pi_{0}(KO\phat[2]\sma \Sigma^{4}KO\phat[2];\bZ/2^{\infty})\ar[r]
&\pi_{0}(\Sigma^{4}KO\phat[2];\bZ/2^{\infty})\ar[d]^{\iso}\\
\pi_{0}(\Sigma^{4}KU\phat[2]\sma \Sigma^{-2}KU\phat[2];\bZ/2^{\infty})\ar[r]
&\pi_{0}(\Sigma^{-2}KU\phat[2];\bZ/2^{\infty})\ar[r]_-{\iso}
&\bZ/2^{\infty}
}
\end{equation*}
that we know commutes by Theorem~\ref{thm:dualXiChi}.  The two maps
\[
\pi_{0}(ko\phat[2]\sma \Sigma^{-2}KU\phat[2];\bZ/2^{\infty})\to \bZ/2^{\infty}
\]
in~\eqref{eq:twomid} are the composite of 
\[
\pi_{0}(ko\phat[2]\sma \Sigma^{-2}KU\phat[2];\bZ/2^{\infty})
\to \pi_{0}(KO\phat[2]\sma \Sigma^{-2}KU\phat[2];\bZ/2^{\infty})
\]
with the two maps in~\eqref{eq:twodone}, and therefore coincide,
proving that~\eqref{eq:twomid} commutes.  Likewise these maps are
composite of the map
\[
\pi_{0}(ko\phat[2]\sma \Sigma^{-2}KU\phat[2];\bZ/2^{\infty})\to
\pi_{0}(ko\phat[2]\sma \tau_{\leq -4}\Sigma^{-2}KU\phat[2];\bZ/2^{\infty})
\]
with the maps in~\eqref{eq:twoneed}.  This map is surjective
because the fiber of 
\[
ko\phat[2]\sma \Sigma^{-2}KU\phat[2]\to 
ko\phat[2]\sma \tau_{\leq -4}\Sigma^{-2}KU\phat[2]
\]
is $ko\phat[2]\sma \Sigma^{-2}ku\phat[2]$ and 
\[ % e.g., by AH since H_{1}(ko\phat[2])=0
\pi_{-1}(ko\phat[2]\sma \Sigma^{-2}ku\phat[2];\bZ/2^{\infty})=0.
\]
It follows that~\eqref{eq:twoneed} commutes as well.
\end{proof}

We have 
\[
F(K(\bS))\simeq F(\Fib(\xi))\phat[2]\sma \Sigma^{-1}M_{\bZ/2^{\infty}}
\]
from~\eqref{eq:MainFib} and the previous theorem gives a weak equivalence 
\[
\Sigma F(\Fib(\xi))\sma \Sigma^{-1}M_{\bZ/2^{\infty}}\simeq \Sigma^{-1}I_{\bZ/2^{\infty}}(\Cof(\chi)).
\]
Since $K(\bZ)\phat[2]\simeq \Fib(\chi)\simeq \Sigma^{-1}\Cof(\chi)$,
we get 
\[
F(K(\bS))\simeq \Sigma^{-3}I_{\bZ/2^{\infty}}(K(\bZ)\phat[2])\simeq \Sigma^{-3}I_{\bZ/2^{\infty}}(K(\bZ)).
\]
This proves Theorem~\ref{main:FS} in the case $p=2$.  

%%%%%%%%%%%%%%%%%%%%%%%%%%%%%%%%%%%%%%%%%%%%%%%%%%%%%%%%%%%%%%%%%%%%%%%%
\section{The fiber of the chromatic completion map on 
\texorpdfstring{$K(\bS_{(p)})$}{K(S(p))}}
\label{sec:Sploc}

The Quillen localization sequence gives a fiber sequence 
\[
\bigvee_{\ell\neq p} K(\bF_{\ell})\to K(\bZ)\to K(\bZ_{(p)})\to \Sigma \dotsb
\]
where the wedge is indexed on the prime numbers $\ell\neq p$.
Because the maps $TC(\bS)\to TC(\bS_{(p)})$ and $TC(\bZ)\to
TC(\bZ_{(p)})$ become weak equivalences
after $p$-completion, the Dundas-Goodwillie-McCarthy homotopy
cartesian square implies an
analogous fiber sequence for $K(\bS)$ and 
$K(\bS_{(p)})$ after $p$-completion. 

\begin{prop}\label{prop:SQseq}
There exists a fiber sequence
\[
\bigl(\bigvee_{\ell\neq p} K(\bF_{\ell})\bigr)\phat\to K(\bS)\phat\to K(\bS_{(p)})\phat\to \Sigma \dotsb 
\]
where the map 
\[
K(\bS_{(p)})\phat\to 
\Sigma \bigl(\bigvee_{\ell\neq p} K(\bF_{\ell})\bigr)\phat. 
\]
is the $p$-completion of the composite of the map $K(\bS_{(p)})\to
K(\bZ_{(p)})$ and the map in the Quillen localization sequence.
\end{prop}

We apply this to deduce the following proposition.
\iffalse
Because we know from~\eqref{eq:MainFib} that $F(K(\bS))$ and
$F(K(\bS_{(p)}))$ are rationally trivial, the fiber of the map
$F(K(\bS))\to F(K(\bS_{(p)}))$ is rationally trivial.  We will see
below that $F(K(\bF_{\ell}))$ is rationally trivial.  Putting this
together with the previous proposition, we get the following one.
\fi

\begin{prop}\label{prop:FKSploc}
There exists a fiber sequence
\[
\bigvee_{\ell\neq p} F(K(\bF_{\ell}))\to F(K(\bS))\to F(K(\bS_{(p)}))\to \Sigma \dotsb 
\]
where the connecting map is induced by connecting the map in
Proposition~\ref{prop:SQseq}. 
\end{prop}

\begin{proof}
Applying Propositions~\ref{prop1} and~\ref{prop4} to the previous
proposition, we get a fiber sequence
\[
F\biggl(\bigvee_{\ell\neq p} K(\bF_{\ell})\biggr)\to F(K(\bS))\to F(K(\bS_{(p)}))\to \Sigma \dotsb.
\]
Since $K(\bF_{\ell})_{(p)}\to L_{1}K(\bF_{\ell})$ is at least
$0$-truncated, $F(K(\bF_{\ell}))$ is weakly equivalent to this
homotopy fiber (by Proposition~\ref{prop3}), and the map
\[
\bigvee_{\ell\neq p} F(K(\bF_{\ell}))\to
F\biggl(\bigvee_{\ell\neq p} K(\bF_{\ell})\biggr)
\]
is a weak equivalence since $L_{1}$ commutes with arbitrary
coproducts.
\end{proof}

We identified $F(K(\bS))$ as $\Sigma^{-3}I_{\bZ/p^{\infty}}K(\bZ)$ in
the previous sections, so to prove Theorem~\ref{main:FSploc} we need
to identify $F(K(\bF_{\ell}))$ and a map $F(K(\bF_{\ell}))\to
F(K(\bS))$ that fits in a fiber sequence in the form in
Proposition~\ref{prop:FKSploc}. 

To identify $F(K(\bF_{\ell}))$, we
use the usual identification of $L_{K(1)}K(\bF_{\ell})$ (for $\ell\neq
p$) as the homotopy fixed points of the action of $\psi^{\ell}$ on
$KU\phat$, or equivalently, as the homotopy fiber of the map
\[
KU\phat\overto{\psi^{\ell}-1}KU\phat.
\]
Then $K(\bF_{\ell})\phat$ is the connective cover and fits into a
fiber sequence of spectra
\[
K(\bF_{\ell})\phat\to ku\phat\overto{\psi^{\ell}-1}\Sigma^{2}ku\phat\to \Sigma \dotsb.
\]
Since $ku\phat$ is the homotopy fiber of the map $KU\phat\to
\tau_{<0}KU\phat$, and $KU\phat$ is $L_{1}$-local, the
$L_{1}$-localization of $ku\phat$ is the fiber of the map 
$KU\phat\to L_{1}(\tau_{<0}KU\phat)$.  Since 
\[
(L_{1}(\tau_{<0}KU\phat))\phat\simeq L_{K(1)}(\tau_{<0}KU\phat)\simeq *,
\]
we have that
\[
L_{1}(\tau_{<0}KU\phat)\simeq L_{1}(\tau_{<0}KU\phat)_{\bQ}\simeq (\tau_{<0}KU\phat)_{\bQ}
\simeq \Sigma^{-2}H\bQ\phat\vee \Sigma^{-4}H\bQ\phat\vee\dotsb.
\]
The map $\tau_{<0}KU\phat \to L_{1}(\tau_{<0}KU\phat)$ induces
rationalization on $\pi_{*}$, so we get that $ku\phat \to
L_{1}(ku\phat)$ is an isomorphism on homotopy groups in degrees $-2$
and up, and 
\[
\pi_{q}(L_{1}(ku\phat))\iso
\begin{cases}
\bZ/p^{\infty}&q\leq -3\text{ and odd}\\
0&q\leq -3\text{ and even}.
\end{cases}
\]
We see from this that the map $K(\bF_{\ell})\phat\to
L_{1}(K(\bF_{\ell})\phat)$ is an isomorphism on homotopy groups in
degrees $-1$ and up, with 
\[
\pi_{q}L_{1}(K(\bF_{\ell})\phat)\iso
\begin{cases}
\bZ/p^{\infty}&q\leq -2\text{ and even}\\
0&q\leq -2\text{ and odd}.
\end{cases}
\]
In particular, the fiber of $K(\bF_{\ell})\phat\to
L_{1}(K(\bF_{\ell})\phat)$ is
$(-3)$-truncated. Propositions~\ref{prop3} and~\ref{prop4} show that
the map $K(\bF_{\ell})\to L_{1}(K(\bF_{\ell}))$
is also $(-3)$-truncated and identify $F(K(\bF_{\ell}))$ as the fiber
of this map, giving the calculation
\[
\pi_{q}F(K(\bF_{\ell}))\iso\begin{cases}
\bZ/p^{\infty}&q\leq -3\text{ and odd}\\
0&\text{otherwise}.
\end{cases}
\]
In particular, we see that $F(K(\bF_{\ell}))$
is rationally trivial and (applying $p$-completion) we get weak equivalences
\begin{gather*}
K(\bF_{\ell})\phat \overto{\simeq}\tau_{\geq 0}(L_{K(1)}(K(\bF_{\ell})))\\
\Sigma F(K(\bF_{\ell}))\phat \overfrom{\simeq} \tau_{\leq -1}(L_{K(1)}K(\bF_{\ell})).
\end{gather*}
It will be slightly more convenient below to work with the suspension
of the inverse of the latter weak equivalence,
\begin{equation}\label{eq:FKFell}
\Sigma^{2} F(K(\bF_{\ell}))\phat \overto{\simeq} \tau_{\leq 0}(\Sigma L_{K(1)}K(\bF_{\ell})).
\end{equation}

As in the previous section, we reformulate the previous calculation in terms of an Anderson
duality pairing.  The identification of $L_{K(1)}K(\bF_{\ell})$ as the
homotopy fixed points of the action of $\psi^{\ell}$ on $KU\phat$
gives a canonical isomorphism
\[
v_{\ell}\colon \pi_{-1}(L_{K(1)}K(\bF_{\ell});\bZ/p^{\infty})\iso
H^{1}(\bZ;\pi_{0}(KU\phat;\bZ/p^{\infty}))\iso \Hom(\bZ,\bZ/p^{\infty})\iso \bZ/p^{\infty}.
\]
(Although Quillen's identification of $K(\bF_{\ell})$ depends on
choices, $v_{\ell}$ does not, since
$\pi_{0}(KU\phat;\bZ/p^{\infty})\iso \pi_{0}(KU\phat)\otimes
\bZ/p^{\infty}$ and the identification of $\pi_{0}(KU\phat)\iso
\bZ\phat$ is canonical, using the unit of $K(\bF_{\ell}$).)
The isomorphism $v_{\ell}$ and the module multiplication 
\[
L_{K(1)}K(\bF_{\ell})\sma \Sigma L_{K(1)}K(\bF_{\ell})\to 
\Sigma L_{K(1)}K(\bF_{\ell})
\]
define a homomorphism
\begin{equation}\label{eq:ADFl}
\pi_{0}(L_{K(1)}K(\bF_{\ell})\sma \Sigma L_{K(1)}K(\bF_{\ell});\bZ/p^{\infty})\to \bZ/p^{\infty}.
\end{equation}

\begin{prop}\label{prop:pADFl}
The homomorphism in~\eqref{eq:ADFl} provides an Anderson duality pairing on $L_{K(1)}(K(\bF_{\ell}))$, $\Sigma
L_{K(1)}(K(\bF_{\ell}))$.
\end{prop}

\begin{proof}
The homotopy fixed point spectral sequence gives an identification
\begin{align*}
\pi_{2q-\epsilon}(L_{K(1)}K(\bF_{\ell});\bZ\phat)&\iso
H^{\epsilon}(\bZ;\bZ\phat(q))\\
\pi_{2q-\epsilon}(L_{K(1)}K(\bF_{\ell});\bZ/p^{\infty})&\iso
H^{\epsilon}(\bZ;\bZ/p^{\infty}(q))\qquad \epsilon=0,1
\end{align*}
where the generator acts by multiplication by $\ell^{q}$ on
$\bZ\phat(q)$ and $\bZ/p^{\infty}(q)$.  The cup product on group
cohomology with coefficients
\[
H^{i}(\bZ;\bZ\phat(q))\otimes H^{j}(\bZ;\bZ/p^{\infty}(r))\to
H^{i+j}(\bZ;\bZ/p^{\infty}(q+r))
\]
converges to the product on homotopy groups and the
isomorphisms in Definition~\ref{defn:BCdcrit} follow from
standard calculations in homological algebra.
\end{proof}

Since the $K(\bF_{\ell})\phat$-module structure on $\Sigma
L_{K(1)}K(\bF_{\ell})$ restricts to a $K(\bF_{\ell})\phat$-module
structure on 
\[
\tau_{\geq 1}\Sigma L_{K(1)}K(\bF_{\ell})\simeq \Sigma K(\bF_{\ell})\phat,
\]
we get a $K(\bF_{\ell})\phat$-module structure on the cofiber
$\tau_{\leq 0}\Sigma L_{K(1)}K(\bF_{\ell})$, 
\[
K(\bF_{\ell})\phat \sma \tau_{\leq 0}(\Sigma
(L_{K(1)}K(\bF_{\ell})))\to \tau_{\leq 0}(\Sigma
(L_{K(1)}K(\bF_{\ell}))).
\]
Composing the induced map on $\bZ/p^{\infty}$ homotopy groups with
$v_{\ell}$ and inspecting the resulting pairing on homotopy groups, we
get an Anderson duality pairing.

\begin{prop}\label{prop:ADFl}
The homomorphism
\[
\pi_{0}(K(\bF_{\ell})\phat \sma \tau_{\leq 0}(\Sigma
(L_{K(1)}K(\bF_{\ell})));\bZ/p^{\infty})\to \bZ/p^{\infty}
\]
induced by the pairing above and $v_{\ell}$ gives an Anderson duality pairing
on $K(\bF_{\ell})\phat$, $\tau_{\leq 0}(\Sigma
(L_{K(1)}K(\bF_{\ell})))$. The weak equivalence~\eqref{eq:FKFell} 
\[
\Sigma^{2}F(K(\bF_{\ell}))\simeq 
\tau_{\leq 0}(\Sigma (L_{K(1)}K(\bF_{\ell})))
\]
then induces an Anderson duality pairing on
$K(\bF_{\ell})\phat$, $\Sigma^{2}F(K(\bF_{\ell}))$.  In particular, we
have a weak equivalence  
\[
F(K(\bF_{\ell}))\simeq \Sigma^{-3}(I_{\bZ/p^{\infty}}(K(\bF_{\ell})\phat))
\simeq \Sigma^{-3}(I_{\bZ/p^{\infty}}(K(\bF_{\ell}))).
\]
\end{prop}

\begin{proof}
Writing $X=K(\bF_{\ell})$ and $Y=\Sigma(L_{K(1)}K(\bF_{\ell}))$, by
construction, the diagram  
\[
\xymatrix{%
\pi_{0}(X\phat \sma Y;\bZ/p^{\infty})
\ar[r]\ar[d]
&\pi_{0}(L_{K(1)}X\sma Y;\bZ/p^{\infty})
\ar[d]
\\
\pi_{0}(X\phat\sma \tau_{\leq 0}Y;\bZ/p^{\infty})
\ar[r]
&\bZ/p^{\infty}
}
\]
commutes and therefore the resulting diagrams
\begin{equation*}\label{eq:explicit}\tag{*}
\begin{gathered}
\xymatrix@-1pc{%
\pi_{q}(X\phat;\bZ/p^{\infty})\ar[r]\ar[d]
&\Hom(\pi_{-q}(\tau_{\leq 0}Y),\bZ/p^{\infty})\ar[d]\\
\pi_{q}(L_{K(1)}X;\bZ/p^{\infty})\ar[r]
&\Hom(\pi_{-q}(Y),\bZ/p^{\infty})\\
\pi_{-q}(Y;\bZ/p^{\infty})\ar[r]\ar[d]
&\Hom(\pi_{q}(L_{K(1)}X),\bZ/p^{\infty})\ar[d]\\
\pi_{-q}(\tau_{\leq 0}Y;\bZ/p^{\infty})\ar[r]
&\Hom(\pi_{q}(X\phat),\bZ/p^{\infty})
}
\end{gathered}
\end{equation*}
also commute.  In the discussion above Proposition~\ref{prop:pADFl},
we observed that the map $X\phat\to L_{K(1)}X$ is a connective
cover. The calculation of $\pi_{*}L_{1}(X\phat)$ there shows that
$\pi_{-1}L_{K(1)}X\iso \bZ\phat$ is torsion free, so
\[
X\phat\sma M_{\bZ/p^{\infty}}\to L_{K(1)}X\sma M_{\bZ/p^{\infty}}
\]
is also a connective cover.  This same calculation shows that
$\pi_{0}Y\iso \bZ\phat$ is torsion free, so the map 
\[
Y\sma M_{\bZ/p^{\infty}}\to (\tau_{\leq 0}Y)\sma M_{\bZ/p^{\infty}}
\]
induced by Whitehead $\tau_{\leq 0}$ truncation for $Y$
is itself a Whitehead $\tau_{\leq 0}$ truncation.
Since the maps in~\eqref{eq:explicit} are
isomorphisms for $L_{K(1)}X, Y$, restricting to $q\geq 0$, we see that
they are isomorphisms for $X\phat, \tau_{\leq 0}Y$.
\end{proof}

To finish the proof of Theorem~\ref{main:FSploc}, we need to identify the maps
\[
F(K(\bF_{\ell}))\to F(K(\bS))\simeq F(\Fib(\trc_{\bZ})).
\]
These maps were defined in Proposition~\ref{prop:FKSploc} using the
natural equivalence $F(-)\to F((-)\phat)$ of Proposition~\ref{prop4}
and the fiber sequence of Proposition~\ref{prop:SQseq}.
Proposition~\ref{prop:SQseq} defines the map 
\[
F(K(\bF_{\ell})\phat)\to F(\Fib(\trc_{\bZ})\phat).
\]
as the induced map on $F(-)$ of the map $K(\bF_{\ell})\phat\to
\Fib(\trc_{\bZ})\phat$ arising from the
comparison of fiber sequences
\begin{equation}\label{eq:Flsquare}
\begin{gathered}
\xymatrix@C-1pc{%
(\mathrm{Fiber})\ar[r]\ar[d]_{\simeq}&
\Fib(\trc_{\bZ})\phat\ar[r]\ar[d]&\Fib(\trc_{\bZ[\tfrac{1}{\ell}]})\phat\ar[d]\ar[r]&\Sigma (\mathrm{Fiber})\ar[d]^{\simeq}\\
K(\bF_{\ell})\phat\ar[r]&
K(\bZ)\phat\ar[r]&K(\bZ[\tfrac{1}{\ell}])\phat\ar[r]&\Sigma K(\bF_{\ell})\phat.
}
\end{gathered}
\end{equation}
around the homotopy cartesian square in the center.
This comparison of fiber sequences lifts to the derived category of
$K(\bZ)\phat$-modules and
we can therefore construct the map $K(\bF_{\ell})\phat\to
\Fib(\trc_{\bZ})\phat$
in the derived category of $K(\bZ)\phat$-modules,
where the $K(\bZ)\phat$-module structure on $K(\bF_{\ell})\phat$
comes from the map of commutative ring spectra $K(\bZ)\to
K(\bF_{\ell})$.   As a consequence, the map 
\[
F(K(\bF_{\ell}))\phat\to F(\Fib(\trc_{\bZ}))\phat
\]
lifts to the derived category of $K(\bZ)\phat$-modules.  Thus, the
following diagram commutes.
\[
\xymatrix@R-1pc{%
K(\bZ)\phat\sma \Sigma^{2}F(K(\bF_{\ell}))\phat\ar[r]\ar[dd]\ar[rd]
&K(\bF_{\ell})\phat\sma \Sigma^{2}F(K(\bF_{\ell}))\phat \ar[d] \\
&\Sigma^{2}F(K(\bF_{\ell}))\phat\ar[d]\\
K(\bZ)\phat\sma \Sigma^{2}F(\Fib(\trc_{\bZ}))\phat\ar[r]
&\Sigma^{2}F(\Fib(\trc_{\bZ}))\phat
}
\]
This diagram almost asserts Anderson duality between the usual map of 
$K$-theory ring spectra
$K(\bZ)\phat\to K(\bF_{\ell})\phat$ and $\Sigma^{2}$ of the
map $F(K(\bF_{\ell}))\to F(\Fib(\trc_{\bZ}))$ in question; such an
assertion would follow from the diagram
\[
\xymatrix@R-1pc{%
\pi_{-2}(F(K(\bF_{\ell}));\bZ/p^{\infty})\ar[r]^-{\iso}\ar[dd]
&\pi_{-1}(L_{K(1)}K(\bF_{\ell});\bZ/p^{\infty})\ar[dd]\ar[dr]^-{v_{\ell}}_(.6){\iso}\\
&&\bZ/p^{\infty}\\
\pi_{-2}(F(\Fib(\trc_{\bZ}));\bZ/p^{\infty})\ar[r]_-{\iso}
&\pi_{-1}(L_{K(1)}\Fib(\trc_{\bZ});\bZ/p^{\infty})\ar[ur]_-{u_{\bZ}}^(.6){\iso}
}
\]
commuting.  The following lemma, proved at the end of the section, 
asserts that it commutes up to multiplying by a unit in $\bZ\phat$.
(In fact, for the isomorphism $u_{\bZ}$ chosen in \cite[\S1]{BM-ktpd} in the
case $p>2$ and in the discussion following Conjecture~\ref{conj:ktpd}
in the case $p=2$, a careful analysis of the proof shows that the unit
is $\pm1$, with the sign depending on the convention for the attaching
map in the Quillen localization sequence, which was not specified
above.) 

\begin{lem}\label{lem:Jtricks}
The map $L_{K(1)}K(\bF_{\ell})\to L_{K(1)}\Fib(\trc_{\bZ})$ induces an
isomorphism on $\pi_{-1}(-;\bZ/p^{\infty})$.
\end{lem}

Combining the lemma with the observations in the paragraph that
proceeds it, we deduce the following theorem.  In the case $p=2$, the
theorem uses Conjecture~\ref{conj:ktpd} (resolved in~\cite[1.2]{Cho-ktpd}).

\begin{thm}\label{thm:FSplocFinal}
%Assume either $p=2$ and Conjecture~\ref{conj:ktpd} holds or $p>2$.
Under the weak equivalences
\begin{gather*}
F(K(\bF_{\ell}))\simeq \Sigma^{-3}I_{\bZ/p^{\infty}}K(\bF_{\ell})\\
F(\Fib(\trc_{\bZ}))\simeq \Sigma^{-3}I_{\bZ/p^{\infty}}K(\bZ)
\end{gather*}
of Theorem~\ref{main:FS} and Proposition~\ref{prop:ADFl}, the map
$F(K(\bF_{\ell}))\to F(\Fib(\trc_{\bZ}))$ in
Proposition~\ref{prop:FKSploc} is $\Sigma^{-3}$ of a unit (in $\bZ\phat$)
times the $\bZ/p^{\infty}$ Brown-Comenetz dual of the usual map of
$K$-theory spectra $K(\bZ)\to K(\bF_{\ell})$.
\end{thm}

Theorem~\ref{main:FSploc} is an immediate consequence: replacing the
maps in Proposition~\ref{prop:FKSploc} by unit multiples, we still get
a weak equivalence of $F(K(\bS_{(p)}))$ with the cofiber.

It remains to prove Lemma~\ref{lem:Jtricks}.  Before beginning the
proof, it is useful to have the following fact, which holds somewhat
more generally than the statement below; it is part of (and indeed preceded) the
general affirmation of the Quillen-Lichtenbaum conjecture, but not
easy to find a direct citation for in the $p=2$ case.

\begin{prop}\label{prop:pione}
Let $R\subset \bQ$ with $1/p\in R$, or let $R=\bQ\phat$. 
Then $\pi_{1}(K(R)\phat)\to
\pi_{1}(L_{K(1)}K(R))$ is an isomorphism.
\end{prop}

\begin{proof}
In the case $R=\bZ[\tfrac1p]$ or $R=\bQ\phat$, this is
\cite[2.2]{Rognesp}. In remaining cases, $\bZ[\tfrac1p]\subset
R\subset \bQ$; let $S$ be the set of prime numbers $\ell\neq p$ with
$1/\ell\in R$.  Since $\bZ[\tfrac1p]$ is a PID, 
\[
R=\bZ[\tfrac1p][1/\ell \mid \ell \in S]
\]
and we have the Quillen localization sequence
\[
\bigvee_{\ell\in S}K(\bF_{\ell})\to K(\bZ[\tfrac1p])
\to K(R)\to \Sigma\dotsb.
\]
We look at the comparison map from the $p$-completion to the
$L_{K(1)}$-localization.  The map $K(\bZ[\tfrac1p])\phat\to K(R)\phat$ is
an isomorphism on $\pi_{0}$ (since these are both PIDs), and this
implies that the map
\[
\pi_{1}(K(R)\phat)\to \pi_{0}\bigl(\bigl(\bigvee K(\bF_{\ell})\bigr)\phat\bigr)
\]
is surjective.  Each map $K(\bF_{\ell})_{(p)}\to L_{1}K(\bF_{\ell})$
is an isomorphism on $\pi_{q}$ for $q\geq -1$, so the map
\[
\pi_{q}\bigl(\bigl(\bigvee K(\bF_{\ell})\bigr)\phat\bigr)
\to \pi_{q}\bigl(L_{K(1)}\bigl(\bigvee K(\bF_{\ell})\bigr)\bigr)
\]
is an isomorphism for $q\geq 0$.  In particular, it follows that the
map 
\[
\pi_{1}(L_{K(1)}K(R))\to \pi_{0}\bigl(L_{K(1)}\bigl(\bigvee K(\bF_{\ell})\bigr)\bigr)
\]
is surjective.  We have already observed that the map 
\[
\pi_{1}(K(\bZ[\tfrac1p])\phat)\to \pi_{1}(L_{K(1)}K(\bZ[\tfrac1p]))
\]
is an isomorphism.  The five lemma now shows that the map
\[
\pi_{1}(K(R)\phat)\to \pi_{1}(L_{K(1)}K(R))
\]
is an isomorphism.
\end{proof}

We also use the following result.  In the form we state it, it follows
from the Quillen localization sequence for inverting $p$ and the fact
that $K(\bF_{p})\phat\simeq H\bZ\phat$ (and therefore the $K$-theory
of any $\bF_{p}$-algebra has trivial
$K(1)$-localization).  Much more general statements hold by
Corollary~C of~\cite{LMMT} and the main theorem
of~\cite{BhattClausenMathew-RemarksK1}. 

\begin{prop}\label{prop:pidinvp}
Let $R$ be a PID in which $p\neq 0$.  Then the map $R\to R[1/p]$
induces an equivalence of $K(1)$-local $K$-theory spectra
$L_{K(1)}K(R)\to L_{K(1)}K(R[1/p])$.
\end{prop}

We now prove Lemma~\ref{lem:Jtricks}.

\begin{proof}[Proof of Lemma~\ref{lem:Jtricks}]
Because $L_{K(1)}K(\bF_{\ell})$ and $L_{K(1)}\Fib(\trc_{\bZ})$ both
have $\pi_{-1}$ isomorphic to $\bZ\phat$ and $\pi_{-2}$ torsion free
($0$ for $L_{K(1)}K(\bF_{\ell})$ and $\bZ\phat$ for
$L_{K(1)}\Fib(\trc_{\bZ})$), it suffices to check that the map induces
an isomorphism on $\pi_{-1}$ (without coefficients). 

Before starting, we fix the following notation: write
$\ell=\omega r^{\eta p^{n}}$ where: 
\begin{enumerate}
\item For $p>2$, $r\in (\bZ\phat)^{\times}$ is a
topological generator of the subgroup $U_{1}\subset
(\bZ\phat)^{\times}$ congruent to $1$ mod $p$, $n$
is a positive integer, $\eta$ is a $p$-adic unit, and $\omega$ is a
$(p-1)$st root of unity. 
\item For $p=2$, $r\in (\bZ\phat[2])^{\times}$ is a
topological generator of the subgroup $U_{1}\subset
(\bZ\phat[2])^{\times}$ congruent to $1$ mod $4$, $n$
is a positive integer, $\eta$ is a $p$-adic unit, and $\omega=\pm1$.
\end{enumerate}
(The element $r^{\eta p^{n}}\in U_{1}$ is by definition the element
that maps to $\eta p^{n}$ under the isomorphism $U_{1}\iso \bZ\phat$
induced by the choice of the topological generator $r$.)

Now consider the composite map obtained by precomposing the suspension
of the map in question with the $L_{K(1)}$-localization of the map in Quillen's
localization sequence for $\bZ\to \bZ[\tfrac1\ell]$,
\[
L_{K(1)}K(\bZ\relax[\tfrac1\ell\relax])
\to \Sigma L_{K(1)}K(\bF_{\ell})\to \Sigma L_{K(1)}\Fib(\trc_{\bZ}).
\]
This composite is (up to a possible sign) the connecting map in the
Mayer-Vietoris (co)fiber sequence associated the homotopy cartesian
square~\eqref{eq:Flsquare}, so coincides with the composite
\[
L_{K(1)}K(\bZ[\tfrac1\ell])\to L_{K(1)}TC(\bZ[\tfrac1\ell])\simeq
L_{K(1)}TC(\bZ)\to \Sigma L_{K(1)}\Fib(\trc_{\bZ}).
\]
Using the isomorphisms from Propositions~\ref{prop:pione} and~\ref{prop:pidinvp}
\[
\pi_{1}((K(\bZ[\tfrac1{p\ell}]))\phat)\overto{\iso}
\pi_{1}(L_{K(1)}K(\bZ[\tfrac1{p\ell}]))\overfrom{\iso}
\pi_{1}(L_{K(1)}K(\bZ[\tfrac1\ell])),
\]
we have a canonical isomorphism 
\[
\pi_{1}(L_{K(1)}K(\bZ[\tfrac1\ell]))\iso (\bZ[\tfrac1{p\ell}]^{\times})\phat.
\]
The element $\ell\in (\bZ[\tfrac1{p\ell}]^{\times})\phat$ then
specifies a map
\[
\Sigma J\phat=\Sigma L_{K(1)}\bS\to L_{K(1)}K(\bZ[\tfrac1\ell]).
\]
Since $\pi_{0}(\Sigma J\phat)\iso \bZ\phat$, to complete the argument
it suffices to show that the composite maps
\begin{gather*}
\pi_{0}(\Sigma J\phat)\to \pi_{0}(\Sigma L_{K(1)}K(\bF_{\ell}))\\
\pi_{0}(\Sigma J\phat)\to \pi_{0}(\Sigma L_{K(1)}\Fib(\trc_{\bZ}))
\end{gather*}
both have cokernels cyclic of order $p^{n}$ (for $n$ the positive
integer in the notation fixed above).

First consider the composite map $\Sigma J\phat\to \Sigma
L_{K(1)}K(\bF_{\ell})$.  Under the canonical isomorphisms (induced by
the inclusion of the unit)
\begin{gather*}
\pi_{1}(\Sigma J\phat)\iso \pi_{0}(J\phat)\iso\bZ\phat\\
\pi_{1}(\Sigma L_{K(1)}K(\bF_{\ell}))\iso \pi_{0}(L_{K(1)}K(\bF_{\ell}))\iso\bZ\phat,
\end{gather*}
the map $\Sigma J\phat\to \Sigma L_{K(1)}K(\bF_{\ell})$ in question
induces on $\pi_{1}$ plus or minus the identity on $\bZ\phat$
(depending on the sign convention in the attaching map in the Quillen
localization sequence that we have not specified).  We have a
canonical isomorphism 
\[
\pi_{-1}(L_{K(1)}K(\bF_{\ell}))\iso 
H^{1}(\bZ;\bZ\phat)\iso \Hom(\bZ,\bZ\phat)\iso \bZ\phat
\]
specified by taking the usual inclusion $\iota\colon \bZ\to\bZ\phat$
as the generator.  We also have an isomorphism
\[
\pi_{-1}(J\phat)\iso H^{1}_{\Gal}((\bZ\phat)^{\times};\bZ\phat)\iso
\Hom_{c}((\bZ\phat)^{\times},\bZ\phat)\iso 
\Hom_{c}(U_{1},\bZ\phat)\iso 
\bZ\phat
\]
(where $\Hom_{c}$ denotes continuous homomorphisms of profinite
groups) specified by taking as the topological generator of
$\Hom_{c}((\bZ\phat)^{\times},\bZ\phat)$ the unique continuous
homomorphism $\phi_{r}$ sending $r$ to $1$.  The map 
\[
\pi_{-1}(J\phat) \iso \pi_{0}(\Sigma J\phat)
\to
\pi_{0}(\Sigma L_{K(1)}K(\bF_{\ell}))\iso
\pi_{-1}(L_{K(1)}K(\bF_{\ell}))
\]
is the map induced by the continuous homomorphism $\bZ\to
(\bZ\phat)^{\times}$ sending the generator to $\ell\in
(\bZ\phat)^{\times}$.  Since $\ell=\omega r^{\eta p^{n}}$, 
\[
\phi_{r}(\ell)=\phi_{r}(\omega)+\eta p^{n}\phi_{r}(r)=\eta p^{n}
\]
($\phi_{r}(\omega)=0$ since $\omega$ is torsion and $\bZ\phat$
is torsion free). The resulting map
\[
\Hom(\bZ,\bZ\phat)\to \Hom_{c}((\bZ\phat)^{\times},\bZ\phat)
\]
therefore sends $\phi_{r}$ to
$\eta p^{n}\iota$, and since $\eta$ is a unit, we see that the cokernel is
cyclic of order $p^{n}$ as claimed.
\iffalse
This calculation can also be done concretely as follows.
Writing $L\phat$ for the Adams summand of $KU\phat$ if $p$ is odd or
$L\phat=KO\phat$ if $p=2$, then $J\phat$ is the homotopy fixed points
of the action of $\psi^{r}$ on $L\phat$ (for $r$ in notation
fixed above), and we have a commuting diagram of fiber sequences
\[
\xymatrix{%
J\phat\ar[r]\ar[d]&L\phat\ar[r]^{\psi^{r}-1}\ar[d]_{\mathrm{inc}}
&L\phat\ar[d]^{\id+\psi^{r}+\dotsb+\psi^{r^{n}-1}}\\
L_{K(1)}K(\bF_{\ell})\ar[r]&KU\phat\ar[r]_{\psi^{\ell}-1}&KU\phat.
}
\]
The map on the right induces on $\pi_{0}\iso\bZ\phat$ multiplication
by $p^{n}$.  A map out of $J\phat$ is specified by
the image of the fundamental class in $\pi_{0}$, so the suspension of this map $J\phat\to
L_{K(1)}K(bF_{\ell})$ is the map $\Sigma J\phat\to \Sigma
L_{K(1)}K(bF_{\ell})$ in question, and it follows that the cokernel of
the induced map on $\pi_{0}$ is cyclic of order $p^{n}$ as claimed.
\fi

For the map $\Sigma J\phat\to \Sigma L_{K(1)}\Fib(\trc_{\bZ})$, we
look at the factorization $\Sigma J\phat\to L_{K(1)}TC(\bZ)$; write
$a_{\ell}$ for this map. The isomorphisms 
\begin{multline*}
\pi_{1}(L_{K(1)}TC(\bZ))\overto{\iso} \pi_{1}(L_{K(1)}TC(\bZ\phat))
\overfrom{\iso} \pi_{1}(L_{K(1)}K(\bZ\phat))\\
\overto{\iso} \pi_{1}(L_{K(1)}K(\bQ\phat))\overfrom{\iso}
\pi_{1}(K(\bQ\phat)\phat)
\end{multline*}
from~\cite[Add.~6.2]{HM2}, 
[ibid., Thm.~D], Proposition~\ref{prop:pidinvp}, and
Proposition~\ref{prop:pione} (in that order)
identify $\pi_{1}(L_{K(1)}TC(\bZ))$
as $((\bQ\phat)^{\times})\phat$.
By construction, $a_{\ell}$ sends the fundamental class
of $\pi_{1}(\Sigma J\phat)$ to $\ell$ under this
identification.  Then for any $\nu \in (\bQ\phat)^{\times}$, we have a
map $a_{\nu}\colon \Sigma J\phat\to L_{K(1)}TC(\bZ)$ that 
sends the fundamental class in $\pi_{1}(\Sigma J\phat)$ to the image
of $\nu$ in $\pi_{1}L_{K(1)}TC(\bZ)$.   In the case $p>2$,
$\omega$ goes to the identity in  $((\bQ\phat)^{\times})\phat$ since
$((\bQ\phat)^{\times})\phat$ is torsion free, and we
have that $a_{\ell}=a_{r^{\eta p^{n}}}=\eta p^{n}a_{r}$.  In the case $p=2$,
$((\bQ\phat)^{\times})\phat$ is not torsion free, but $\pi_{1}(\Sigma
L_{K(1)}\Fib(\trc_{\bZ}))$ is torsion free, so
$a_{\ell}$ and $\eta p^{n}a_{r}$ become the same after composing
with the map $\Sigma L_{K(1)}TC(\bZ)\to \Sigma
L_{K(1)}\Fib(\trc_{\bZ})$. Since the composite 
\[
J\phat\overto{a_{r}}  L_{K(1)}TC(\bZ)\to \Sigma L_{K(1)}\Fib(\trc_{\bZ})
\]
is an isomorphism on $\pi_{0}$, the map in question
\[
J\phat\overto{a_{\ell}}  L_{K(1)}TC(\bZ)\to \Sigma L_{K(1)}\Fib(\trc_{\bZ})
\]
has cokernel cyclic of order $p^{n}$ as claimed.  This completes
the proof.
\end{proof}

%%%%%%%%%%%%%%%%%%%%%%%%%%%%%%%%%%%%%%%%%%%%%%%%%%%%%%%%%%%%%%%%%%%%%%%%
\section{Proof of Theorems~\ref{main:waldsplit} and~\ref{main:Swaldsplit}}
\label{sec:wspf}

In this section, we prove Theorems~\ref{main:waldsplit}
and~\ref{main:Swaldsplit}, which state that the maps
\begin{gather*}
K(\bS_{(p)})_{(p)}\to \holim_{n} K(L^{f}_{n} \bS)_{(p)}\\
K(\bS)_{(p)}\to \holim_{n} K(L^{p,f}_{n} \bS)_{(p)}
\end{gather*}
are inclusions of wedge summands.  The argument relies on recent work
of Land-Mathew-Meier-Tamme~\cite{LMMT}; we use the following result
that we have stated in substantially weaker form than it appears
there.

\begin{thm}[{Land-Mathew-Meier-Tamme~\cite{LMMT}, Theorem~A}]\label{thm:LMMT}
Let $R\to R'$ be an $L^{p,f}_{n}$-equivalence of ring spectra.  Then 
\[
(L_{n}K(R))\phat\to (L_{n}K(R'))\phat.
\]
is a weak equivalence.
\end{thm}

The maps $\bS_{(p)}\to L^{f}_{n}\bS$ and $\bS\to L^{p,f}_{n}\bS$ are
$L^{p,f}_{n}$-localization maps and in particular
$L^{p,f}_{n}$-equivalences.  We therefore get the following corollary.

\begin{cor}\label{cor:LMMT}
The maps
\begin{gather*}
\holim_{n} (L_{n}K(\bS_{(p)}))\phat \to \holim_{n} (L_{n}K(L^{f}_{n} \bS))\phat\\
\holim_{n} (L_{n}K(\bS))\phat \to \holim_{n} (L_{n}K(L^{p,f}_{n} \bS))\phat
\end{gather*}
are weak equivalences.
\end{cor}

Using the localization maps
\begin{gather*}
\holim_{n} K(L^{f}_{n} \bS)\phat \to
\holim_{n} (L_{n}K(L^{f}_{n} \bS))\phat\\
\holim_{n} K(L^{p,f}_{n} \bS)\phat\to
\holim_{n} (L_{n}K(L^{p,f}_{n} \bS))\phat
\end{gather*}
and the inverse equivalence from Corollary~\ref{cor:LMMT}, we get maps
\begin{equation}\label{eq:psplit}
\begin{gathered}
\holim_{n} K(L^{f}_{n} \bS)\phat\to \holim_{n} (L_{n}K(\bS_{(p)}))\phat\\
\holim_{n} K(L^{p,f}_{n} \bS)\phat\to \holim_{n} (L_{n}K(\bS))\phat
\end{gathered}
\end{equation}
with the property that the composite maps 
\begin{gather*}
K(\bS_{(p)})\phat \to \holim_{n} (L_{n}K(\bS_{(p)}))\phat\\
K(\bS)\phat \to \holim_{n} (L_{n}K(\bS))\phat
\end{gather*}
are the $p$-completions of the chromatic completion maps. 

We also have maps on rationalizations, constructed as follows.  Since
the ring spectrum $\bQ$ is $L^{p,f}_{n}$-local, the maps of ring
spectra $\bS_{(p)}\to \bQ$ and $\bS\to \bQ$ factor through maps
$L^{f}_{n}\bS\to \bQ$ and $L^{p,f}_{n}\bS\to \bQ$.  The Quillen
localization sequence and the fact that the rational $K$-theory of
finite fields is zero except in degree zero imply that the maps
\begin{equation}\label{eq:isoexcept1}
\begin{gathered}
K(\bS_{(p)})_{\bQ}\to K(\bQ)_{\bQ}\\
K(\bS)_{\bQ}\to K(\bQ)_{\bQ}
\end{gathered}
\end{equation}
induce an isomorphism on homotopy groups except in degree $1$; in this
degree, 
$\pi_{1}(K(\bS);\bQ)=0$ and we have a canonical splitting
\[
\pi_{1}(K(\bQ);{\bQ})\iso (\bQ^{\times})\otimes \bQ\iso 
\bigoplus_{\ell\text{ prime}}\!\!\bQ
\;\to\!\!\!\! \bigoplus_{\ell\text{ prime},\,\ell\neq p}\!\!\!\!\!\!\bQ
\iso (\bZ_{(p)}^{\times})\otimes \bQ\iso \pi_{1}(K(\bS_{(p)});{\bQ}).
\]
The splittings on rational homotopy groups specify splittings of
rational spectra and we get maps
\begin{equation}\label{eq:qsplit}
\begin{gathered}
(\holim_{n} K(L^{f}_{n}\bS))_{\bQ}\to K(\bQ)_{\bQ}\to
K(\bS_{(p)})_{\bQ} \overto{\simeq} (\holim_{n} L_{n}K(\bS_{(p)}))_{\bQ}\\
(\holim_{n} K(L^{p,f}_{n}\bS))_{\bQ}\to K(\bQ)_{\bQ}\to
K(\bS)_{\bQ} \overto{\simeq} (\holim_{n} L_{n}K(\bS))_{\bQ}
\end{gathered}
\end{equation}
with the property that the composite maps
\begin{gather*}
K(\bS_{(p)})_{\bQ} \to (\holim_{n} L_{n}K(\bS_{(p)}))_{\bQ}\\
K(\bS)_{\bQ} \to (\holim_{n} L_{n}K(\bS))_{\bQ}
\end{gather*}
are the rationalizations of the chromatic completion maps.

We have constructed a $p$-complete map and a rational map.  The
following theorem, proved in the next section, asserts that they are
compatible with the arithmetic square. 

\begin{thm}\label{thm:arith}
For $S=\bS$ or $\bS_{(p)}$, the diagram in the stable category 
\[
\xymatrix@-1pc{%
(\holim K(L^{p,f}_{n}S))_{(p)}\ar[r]\ar[d]
&\holim K(L^{p,f}_{n}S)\phat\ar[r]^-{\eqref{eq:psplit}}
&\holim L_{n}K(S)\phat\ar[dd]\\
(\holim K(L^{p,f}_{n}S))_{\bQ}\ar[d]_-{\eqref{eq:qsplit}}\\
K(S)_{\bQ}\ar[r]_-{\simeq}&
(\holim L_{n}K(S))_{\bQ}\ar[r]&((\holim L_{n}K(S))\phat)_{\bQ}
}
\]
commutes.  
\end{thm}

As we now explain, the previous theorem completes the proof of 
Theorems~\ref{main:waldsplit} and~\ref{main:Swaldsplit}. 
The previous theorem implies that there exists a map
\begin{equation}\label{eq:ccintsplit}
(\holim K(L^{p,f}_{n}S))_{(p)}\to \holim L_{n}K(S)
\end{equation}
such that the composite maps 
\begin{gather*}
K(S)_{(p)}\to (\holim K(L^{p,f}_{n}S))_{(p)}\to (\holim L_{n}K(S))\phat\\
K(S)_{(p)}\to (\holim K(L^{p,f}_{n}S))_{(p)}\to (\holim L_{n}K(S))_{\bQ}
\end{gather*}
are the canonical ones, i.e., the chromatic completion map followed by
the $p$-completion map and rationalization maps, respectively.  We
first need to observe that the map~\eqref{eq:ccintsplit} factors through
$K(S)_{(p)}$. 

In the case $S=\bS$, this is clear because as the limit of connective
spectra, $\holim K(L^{p,f}_{n}S)$ is $(-2)$-connected and the
$(-2)$-connected cover of $\holim L_{n}K(\bS)$ is $K(\bS_{(p)})_{(p)}$
by Theorem~\ref{main:FS}.  In the case $S=\bS_{(p)}$, the
$(-2)$-connective cover of $\holim L_{n}K(\bS_{(p)})$ has non-trivial
$\pi_{-1}$, and it suffices to see that the map on $\pi_{-1}$ is
trivial.  In the notation of~\ref{notn:F}, the map 
\[
\pi_{-1}(\holim L_{n}K(\bS_{(p)}))\to \pi_{-2}(F(K(\bS_{(p)})))
\]
is injective; by 
Theorem~\ref{main:FS}, $\pi_{-2}F(K(\bS))=0$, and that implies that the map 
\[
F(K(\bS_{(p)})))\to \Sigma \bigvee_{\ell\neq p}F(K(\bF_{\ell}))
\simeq \bigvee_{\ell\neq p}F(\Sigma K(\bF_{\ell}))
\]
in the fibration sequence of Proposition~\ref{prop:FKSploc} is
injective on $\pi_{-2}$.  Since the maps 
\[
L_{1}\Sigma K(\bF_{\ell})\to \Sigma F(\Sigma K(\bF_{\ell}))
\]
are injective on $\pi_{-1}$ and the diagram
\[
\xymatrix{%
\holim L_{n}K(\bS_{(p)})\ar[r]\ar[d]&\Sigma F(K(\bS_{(p)}))\ar[d]\\
L_{1}\Sigma K(\bF_{\ell})\ar[r]&\Sigma F(\Sigma K(\bF_{\ell}))
}
\]
commutes by naturality, it suffices to observe that the maps 
\[
\holim K(L^{f}_{n}\bS)\to L_{1}\Sigma K(\bF_{\ell})
\]
are all zero on $\pi_{-1}$.  But since by construction these maps
factor through $L_{1}K(\bZ_{(p)})$ (see Proposition~\ref{prop:SQseq}),
they factor through $K(L^{f}_{1}\bS)$, which is connective.

The factorization constructed above is a map
\begin{equation}\label{eq:intsplit}
(\holim K(L^{p,f}_{n}S))_{(p)}\to K(S)_{(p)}
\end{equation}
such that the composite maps 
\begin{gather*}
K(S)_{(p)}\to (\holim K(L^{p,f}_{n}S))_{(p)}\to K(S)\phat\\
K(S)_{(p)}\to (\holim K(L^{p,f}_{n}S))_{(p)}\to K(S)_{\bQ}
\end{gather*}
are the $p$-completion and rationalization maps, respectively.  In the
fiber sequence for the arithmetic square for $K(S)_{(p)}$, the
connecting map  
\[
\Sigma^{-1}K(S)_{\bQ}\to K(S)_{(p)}
\]
is the zero map on homotopy groups, because
each homotopy group of $K(\bS)$ is finitely generated and each
homotopy group of $K(\bS_{(p)})$ is the extension by a finitely
generated group of a subgroup of a direct sum of finitely generated
groups.  It follows that the composite map
\begin{equation}\label{eq:comp}
K(S)_{(p)}\to (\holim K(L^{p,f}_{n}S))_{(p)}\to K(S)_{(p)}
\end{equation}
induces the identity map  on homotopy groups and in particular is a
weak equivalence.  Composing the map~\eqref{eq:intsplit} with the inverse of
this weak equivalence, we get a map
\begin{equation*}\label{eq:finalsplit}
(\holim K(L^{p,f}_{n}S))_{(p)}\to K(S)_{(p)}
\end{equation*}
so that the composite map
\[
K(S)_{(p)}\to (\holim K(L^{p,f}_{n}S))_{(p)}\to K(S)_{(p)}
\]
is the identity, completing the proof.

We close the section with a few remarks on the chromatic localization
analogue of Theorem~\ref{main:waldsplit}, specifically what is missing
to prove that the map
\[
K(\bS_{(p)})_{(p)}\to \holim K(L_{n}\bS)_{(p)}
\]
is the inclusion of a wedge summand.  Most of the argument (in this
section and the next) works for $L_{n}$ as it does for $L_{n}^{f}$ except the
Land-Mathew-Meier-Tamme result (Theorem~\ref{thm:LMMT}) that we used
to construct the map~\eqref{eq:psplit}.  For the
chromatic localization version, we need to construct a map 
\begin{equation}\label{eq:cpsplit}
\holim K(L_{n}\bS)\phat\to \holim (L_{n}K(\bS_{(p)}))\phat
\end{equation}
by some other means.  We envision such a construction would go as
follows.  Previous work of the authors~\cite[1.1,2.6]{BMY1} shows that
the map
\[
TC(\bS_{(p)})\phat\to \holim TC(L_{n}\bS_{(p)})\phat\overto{\simeq}
\holim (L_{n}TC(L_{n}\bS_{(p)}))\phat
\]
is a weak equivalence (the latter weak equivalence holding because
$TC(L_{n}\bS_{(p)})$ is $L_{n}$-local).  We also have by the
Land-Mathew-Meier-Tamme result in Theorem~\ref{thm:LMMT} (and the
telescope conjecture in the case $n=1$) that the map
\[
(L_{1}K(\bS_{(p)}))\phat\to (L_{1}K(L_{1}\bS_{(p)}))\phat
\]
is a weak equivalence, which then gives us a map 
\[
(L_{1}K(L_{1}\bS_{(p)}))\phat\to (L_{1}TC(\bS_{(p)}))\phat.
\]
We would then expect (but have not been able to verify) that the
resulting diagram
\begin{equation}\label{eq:stabbingspree}
\begin{gathered}
\xymatrix{%
\holim K(L_{n}\bS)\phat\ar[r]\ar[d]&TC(\bS_{(p)})\phat\ar[d]\\
(L_{1}K(L_{1}\bS_{(p)}))\phat\ar[r]&(L_{1}TC(\bS_{(p)}))\phat
}
\end{gathered}
\end{equation}
should commute where the horizontal are maps as discussed and the
vertical maps are induced by $L_{1}$-localization.  If this is the
case, then using the commutative diagram
\[
\xymatrix{%
(L_{1}K(L_{1}\bS_{(p)}))\phat\ar[d]
&(L_{1}K(\bS_{(p)}))\phat\ar[l]_{\simeq}\ar[d]\ar[r]
&L_{1}TC(\bS_{(p)})\phat\ar[d]\\
(L_{1}K(\bQ))\phat
&(L_{1}K(\bZ_{(p)}))\phat\ar[l]^{\simeq}\ar[r]
&L_{1}TC(\bZ_{(p)})\phat,
}
\]
we get a commutative diagram
\[
\xymatrix{%
\holim K(L_{n}\bS)\phat\ar[r]\ar[d]&TC(\bS_{(p)})\phat\ar[d]\\
(L_{1}K(\bZ_{(p)}))\phat\ar[r]&(L_{1}TC(\bZ_{(p)}))\phat
}
\]
which then entitles us to a map of the form~\eqref{eq:cpsplit}: the
outer square
\[
\xymatrix@C-1pc{%
\holim (L_{n}K(\bS_{(p)}))\phat\ar[r]\ar[d]
&TC(\bS_{(p)})\phat\ar[d]\ar[r]^-{\simeq}
&\holim (L_{n}TC(\bS_{(p)}))\phat\ar[d]\\
(L_{1}K(\bZ_{(p)}))\phat\ar[r]
&(L_{1}TC(\bZ_{(p)}))\phat\ar[r]_{\id}
&(L_{1}TC(\bZ_{(p)}))\phat
}
\]
is homotopy cartesian by the Dundas-McCarthy theorem and the vanishing
of $L_{K(n)}$, $n>1$ on the bottom row.  The inner square is therefore also
homotopy cartesian.  The resulting map~\eqref{eq:cpsplit} 
is not uniquely determined, but for any choice, the composite maps 
\begin{gather*}
K(\bS_{(p)})\phat\to \holim K(L_{n}\bS)\phat\to TC(\bS_{(p)})\phat\\
K(\bS_{(p)})\phat\to \holim K(L_{n}\bS)\phat\to L_{1}K(\bZ_{(p)})\phat
\end{gather*}
are the usual ones.
We can show in the case $p>2$
(applying~\cite[1.1]{BM-KSpi} and an argument like the analysis of the
map~\eqref{eq:comp} above) 
that the map~\eqref{eq:cpsplit} can be chosen so that the composite
map 
\[
K(\bS_{(p)})\phat\to \holim K(L_{n}\bS)\phat\to \holim (L_{n}K(\bS_{(p)}))\phat
\]
is the $p$-completion of the chromatic completion map.  In other
words, the argument for the chromatic localizations reduces in the
case $p>2$ to showing that~\eqref{eq:stabbingspree} commutes.

We note that post composing the diagram~\eqref{eq:stabbingspree} with
the map 
\begin{equation}\label{eq:stabbintheeye}
(L_{1}TC(\bS_{(p)}))\phat\to 
(L_{1}TC(L_{1}\bS_{(p)}))\phat
\end{equation}
(induced by the map $\bS_{(p)}\to L_{1}\bS_{(p)}$) does result in a
commuting diagram, but~\eqref{eq:stabbintheeye} is far from a weak
equivalence.  In fact, the difference can be detected in rational
homotopy groups mod the Serre class of finite dimensional $\bQ\phat$
vector spaces (in the category of all $\bQ\phat$ vector spaces).
Since this is a negative result, we just give an outline.
Starting from
\[
TC(\bS_{(p)})\phat \simeq \bS\phat \vee \Fib(\Sigma
(\Sigma^{\infty}_{+}\bC P^{\infty})\phat\sto \bS\phat),
\]
Ravenel~\cite[9.2]{Ravenel-Periodic} shows that
$\pi_{*}((L_{1}TC(\bS_{(p)}))\phat;\bQ)$ is infinite dimensional over
$\bQ\phat$ in infinitely many even degrees.  On the other hand, basic
Tate vanishing results imply
\[
(L_{1}TC(L_{1}\bS_{(p)}))\phat\simeq (L_{K(1)}\bS)^{h\bT}.
\]
(Although $L_{1}\bS$ is non-connective and the theorem identifying
classic $TC$ and Nikolaus-Scholze $TC$ does not apply, both
constructions give the same result in this case.)  We can
calculate $(L_{K(1)}\bS)^{h\bT}$ as the homotopy fiber of the self map
$\psi^{r}-1$ on $(KU\phat)^{h\bT}$ (for $r$ as above).  We have 
\[
\pi_{*}((KU\phat)^{h\bT})\iso (KU\phat)^{*}(\bC P^{\infty})\iso \bZ\phat{}[u,u^{-1}][[x]]
\]
and $\psi^{r}$ is a $\bZ\phat$-algebra map that acts on $x$ by the
$r$-series and on $u$ by multiplication by $r$.  It follows that
$\psi^{r}-1$ is injective on homotopy groups except in degree 0, where
it has kernel $\bZ\phat$ (generated by the unit).  This shows that
$\pi_{*}((L_{1}TC(L_{1}\bS_{(p)}))\phat;\bQ)$ is zero in the non-zero
even degrees, and one dimensional over $\bQ\phat$ in degree zero.

%%%%%%%%%%%%%%%%%%%%%%%%%%%%%%%%%%%%%%%%
\section{Proof of Theorem~\ref{thm:arith}}\label{sec:arith}

The entirety of this section is devoted to the proof of
Theorem~\ref{thm:arith}. As in the statement, let $S=\bS$ or
$\bS_{(p)}$ and then $L_{n}^{p,f}S$ is $L_{n}^{p,f}\bS$ or
$L_{n}^{f}\bS$ (respectively). 

We need to show that the two maps
\[
(\holim K(L^{p,f}_{n}S))_{(p)}\to ((\holim L_{n}K(S))\phat)_{\bQ}
\]
in the diagram in the statement coincide, and because the spectrum on
the right is rational, this amounts to a rational homotopy group
calculation.  It is equivalent (and notationally lighter) to show that
the two maps 
\[
\pi_{*}(\holim K(L^{p,f}_{n}S))\to
\pi_{*}((\holim L_{n}K(S))\phat;\bQ)
\]
coincide.  The main tool we use is the diagram
\begin{equation}\label{eq:bigarith}%
\begin{gathered}\hspace{-1em}
\xymatrix@-1pc{%
\holim K(L^{p,f}_{n}S)\ar[rrr]\ar[dr]\ar[ddd]
&&&(\holim L_{n}K(S))\phat\ar[ddd]\ar[dl]\\
&K(\bQ)\ar[r]\ar[d]&L_{K(1)}K(\bQ)\ar[d]\\
&K(\bQ)_{\bQ}\ar[r]\ar[dl]&(L_{K(1)}K(\bQ))_{\bQ}\\
K(S)_{\bQ}%(\holim L_{n}K(S))_{\bQ}
&&&((\holim L_{n}K(S))\phat)_{\bQ}\ar[ul]
}\end{gathered}\hspace{-1em}
\end{equation}
which commutes by construction: the
left vertical arrow is essentially defined in~\eqref{eq:qsplit} by the
lefthand trapezoid.  Note the direction of the lower left diagonal
arrow; the first step in the proof is to show that we can make it go
the other way:

\begin{lem}\label{lem:arithfact}
The map $\pi_{*}(\holim K(L^{p,f}_{n}S);\bQ)\to \pi_{*}(K(\bQ);\bQ)$
factors through the image of $\pi_{*}(K(S);\bQ)$.
\end{lem}

We observe below that this implies that the two composite maps
\[
\pi_{*}(\holim K(L^{p,f}_{n}S))\to
\pi_{*}((\holim L_{n}K(S))\phat;{\bQ})
\to \pi_{*}(L_{K(1)}K(\bQ);{\bQ})
\]
coincide.  The following lemma then implies that this gives the bulk of the
rational homotopy group calculation.

\begin{lem}\label{lem:arithinj}
The map $\pi_{*}((\holim L_{n}K(S))\phat;\bQ)\to \pi_{*}(L_{K(1)}K(\bQ);\bQ)$
is an injection for $*\geq 0$.
\end{lem}

Before proving these lemmas, we show how they prove Theorem~\ref{thm:arith}.

\begin{proof}[Proof of Theorem~\ref{thm:arith}]
As discussed above, since the target spectrum is rational, it suffices
to show that the two maps on homotopy groups
\[
\pi_{*}(\holim K(L^{p,f}_{n}S))\to
\pi_{*}((\holim K(L^{p,f}_{n}S))_{(p)})\to 
\pi_{*}((\holim L_{n}K(S))\phat;{\bQ})
\]
coincide.

Since $K(S)_{\bQ}\to K(\bQ)_{\bQ}$ is a split monomorphism, 
Lemma~\ref{lem:arithfact} together with diagram~\eqref{eq:bigarith}
give the commuting diagram
\[
\xymatrix@C-1pc{%
\holim K(L^{p,f}_{n}S)\ar[rr]\ar[d]&&(\holim L_{n}K(S))\phat\ar[d]\\
K(S)_{\bQ}\ar[r]&(L_{K(1)}K(\bQ))_{\bQ}
&((\holim L_{n}K(S))\phat)_{\bQ}.\ar[l] 
}
\]
By naturality, the diagram
\[
\xymatrix@C-1pc{%
K(S)_{\bQ}\ar[r]^-{\simeq}\ar[dr]
&(\holim L_{n}K(S))_{\bQ}\ar[d]\ar[r]
&((\holim L_{n}K(S))\phat)_{\bQ}\ar[dl]\\
&(L_{K(1)}K(\bQ))_{\bQ}
}
\]
also commutes.  It follows that the two composite maps 
\[
\pi_{*}(\holim K(L^{p,f}_{n}S))\to
\pi_{*}((\holim L_{n}K(S))\phat;{\bQ})
\to \pi_{*}((L_{K(1)}K(\bQ));{\bQ})
\]
coincide.
By Theorems~\ref{main:FS} and~\ref{main:FSploc}, we
have that $\pi_{-1}(\holim L_{n}K(S)\phat;\bQ)=0$.
Lemma~\ref{lem:arithinj} and the fact that  $\holim
K(L^{p,f}_{n}S)$ is $(-2)$-connected then imply that the two maps
\[
\pi_{*}(\holim K(L^{p,f}_{n}S))\to
\pi_{*}((\holim L_{n}K(S))\phat;{\bQ})
\]
coincide
\end{proof}

We now prove the lemmas.

\begin{proof}[Proof of Lemma~\ref{lem:arithfact}]
By the remarks surrounding (\ref{eq:isoexcept1}), it suffices to analyze the case $*=1$.  
As reviewed above, we have 
\[
\pi_{1}(K(\bQ);{\bQ})\iso(\bQ^{\times})\otimes \bQ \iso \bigoplus_{\ell\text{ prime}}\bQ,
\]
and we write $\nu_{\ell}\colon \pi_{1}(K(\bQ);{\bQ})\to \bQ$ for the
projection to the $\ell$ factor; this homomorphism is induced by
$\ell$-adic valuation on $\bQ^{\times}$.  The image of
$\pi_{1}K(\bS_{(p)};\bQ)$ is the kernel of $\nu_{p}$ and the image of
$\pi_{1}K(\bS;\bQ)$ is zero, or equivalently, the intersection of the
kernel of $\nu_{\ell}$ for all $\ell$.

In the case $S=\bS$, the map $\holim K(L^{p,f}_{n}\bS)\to K(\bQ)$
factors through $K(L^{p,f}_{0}\bZ)=K(\bZ[1/p])$.  We see from the
commuting diagram
\[
\xymatrix@R-1pc{%
\bZ[1/p]^{\times}\ar[r]^-{\iso}\ar[d]&\pi_{1}(K(\bZ[1/p]))\ar[d]\\
\bQ^{\times}\ar[r]_-{\iso}&\pi_{1}(K(\bQ))
}
\]
that the image of $\pi_{1}(K(\bZ[1/p]);\bQ)$ in $\pi_{1}(K(\bQ);\bQ)$
lands in the kernel of $\nu_{\ell}$ for all $\ell\neq p$.  To see that
the image of $\pi_{1}(\holim K(L^{p,f}_{n}\bS);\bQ)$ also lands in the
kernel of $\nu_{p}$, we note that it lands in the image of
$\pi_{1}(\holim K(L^{p}_{n}\bS);\bQ)$, and this reduces the case
$S=\bS$ to the case $S=\bS_{(p)}$.

We are left to establish the case $S=\bS_{(p)}$. (As above, we write
$L^{f}_{n}\bS$ for $L^{p,f}_{n}S$ in the case $S=\bS_{(p)}$.) Because
$\pi_{0}K(\bS_{(p)})\iso \pi_{0}K(\bQ)\iso \bZ$ is 
torsion-free, $\pi_{1}(K(\bS_{(p)})\phat)$ and $\pi_{1}(K(\bQ)\phat)$
are the left $0$-derived $p$-completions of $\pi_{1}(K(\bS_{(p)}))\iso \bZ_{(p)}^{\times}$ and
$\pi_{1}(K(\bQ))\iso \bQ^{\times}$.  The homomorphism $\nu_{p}$ then extends to a
homomorphism $\nu_{p}\colon \pi_{1}(K(\bQ)\phat;\bQ)\to \bQ\phat$ with kernel
$\pi_{1}(K(\bS_{(p)})\phat;\bQ)$.  We have an analogous homomorphism $\nu_{p}\colon
\pi_{1}(K(\bQ\phat)\phat;\bQ)\to \bQ\phat$ and a commuting diagram of
short exact sequences
\[
\xymatrix@-1pc{%
0\ar[r]
&\pi_{1}(K(\bS_{(p)});\bQ)\ar[r]\ar[d]
&\pi_{1}(K(\bQ);\bQ)\ar[r]^-{\nu_{p}}\ar[d]
&\bQ\ar[r]\ar[d]^{\subset}&0\\
0\ar[r]
&\pi_{1}(K(\bS_{(p)})\phat;\bQ)\ar[r]\ar[d]
&\pi_{1}(K(\bQ)\phat;\bQ)\ar[r]^-{\nu_{p}}\ar[d]
&\bQ\phat\ar[r]\ar[d]^{=}&0\\
0\ar[r]
&\pi_{1}(K(\bZ\phat)\phat;\bQ)\ar[r]
&\pi_{1}(K(\bQ\phat)\phat;\bQ)\ar[r]^-{\nu_{p}}&\bQ\phat\ar[r]&0.
}
\]
To complete the proof, we need to see that the composite of
$\pi_{1}(\holim K(L^{f}_{n}\bS);\bQ)\to \pi_{1}(K(\bQ\phat);\bQ)$ with
$\nu_{p}$ is zero. 

By~\cite[Thm~D]{HM2}, we have a weak equivalence 
\[
K(\bZ\phat)\phat\overto{\simeq}\tau_{\geq 0}(TC(\bZ\phat)\phat).
\]
Together with Proposition~\ref{prop:pidinvp}, this gives weak equivalences
\[
L_{K(1)}TC(\bZ\phat)\overfrom{\simeq}L_{K(1)}K(\bZ\phat)\overto{\simeq}L_{K(1)}K(\bQ\phat),
\]
and the resolved Quillen-Lichtenbaum
conjecture~\cite[Thm.~A]{HMAnnals} for $\bQ\phat$ implies the weak
equivalence
\[
\tau_{\geq 1}K(\bQ\phat)\overto{\simeq}\tau_{\geq 1}L_{K(1)}K(\bQ\phat).
\]
We then get a commuting diagram of short exact sequences
\[
\xymatrix@-1pc{%
0\ar[r]
&\pi_{1}(K(\bZ\phat)\phat;\bQ)\ar[r]\ar[d]^{=}
&\pi_{1}(K(\bQ\phat)\phat;\bQ)\ar[r]^-{\nu_{p}}\ar[d]^{\iso}&\bQ\phat\ar[r]\ar[d]^{=}&0\\
0\ar[r]
&\pi_{1}(K(\bZ\phat)\phat;\bQ)\ar[r]\ar@{<-}[d]^{=}
&\pi_{1}(L_{K(1)}K(\bQ\phat);\bQ)\ar[r]\ar@{<-}[d]^{\iso}
&\bQ\phat\ar[r]\ar@{<-}[d]^{=}&0\\
0\ar[r]
&\pi_{1}(K(\bZ\phat)\phat;\bQ)\ar[r]\ar[d]^{\iso}
&\pi_{1}(L_{K(1)}K(\bZ\phat);\bQ)\ar[r]\ar[d]^{\iso}
&\bQ\phat\ar[r]\ar[d]^{=}&0\\
0\ar[r]
&\pi_{1}(TC(\bZ\phat)\phat;\bQ)\ar[r]
&\pi_{1}(L_{K(1)}TC(\bZ\phat);\bQ)\ar[r]^-{\nu}&\bQ\phat\ar[r]&0
}
\]
By~\eqref{eq:bigarith} and naturality, the diagram
\[
\xymatrix@C-1pc{%
\pi_{1}(\holim K(L^{f}_{n}\bS);\bQ)\ar[rr]\ar[d]
&&\pi_{1}(K(\bQ\phat)\phat;\bQ)\ar[d]^{\iso}\ar[dr]^{\nu_{p}}\\
\pi_{1}((\holim L_{n}K(\bS_{(p)}))\phat;\bQ)\ar[r]\ar[d]
&\pi_{1}(L_{K(1)}K(\bZ\phat);\bQ)\ar[r]^{\iso}\ar[dr]^{\iso}
&\pi_{1}(L_{K(1)}K(\bQ\phat);\bQ)\ar[r]&\bQ\phat\\
\pi_{1}(TC(\bS_{(p)})\phat;\bQ)\ar[r]
&\pi_{1}(TC(\bZ\phat)\phat;\bQ)\ar[r]
&\pi_{1}(L_{K(1)}TC(\bZ\phat);\bQ)\ar[ur]_{\nu}
}
\]
commutes, and we see that the composite 
$\pi_{1}(\holim K(L^{f}_{n}\bS);\bQ)\to \bQ\phat$ factors through
$\pi_{1}(TC(\bZ\phat)\phat;\bQ)$, which is in the kernel of $\nu$.
\end{proof}

\begin{proof}[Proof of Lemma~\ref{lem:arithinj}]
In the case $S=\bS$, consider the following commutative diagram.
\[
\xymatrix{%
K(\bS)\phat\ar[r]\ar[d]&K(\bZ)\phat\ar[d]\\
\holim L_{n}K(\bS)\phat\ar[r]&L_{K(1)}K(\bZ)\ar[r]&L_{K(1)}K(\bQ)
}
\]
We claim that all maps are injections on $\pi_{q}(-;\bQ)$ for $q\geq
0$.  The top horizontal map is a rational equivalence because the
linearization map $K(\bS)\to K(\bZ)$ is a rational
equivalence~\cite[2.3.8ff]{WaldhausenKT} and the 
homotopy groups of $K(\bS)$ and $K(\bZ)$ are finitely
generated~\cite[Theorem~1]{Quillen-FGK}, \cite[1.2]{DwyerAX}.  The
left vertical map is a rational equivalence by Theorem~\ref{main:FS} 
(or Proposition~\ref{prop:integral}).  The established
Quillen-Lichtenbaum conjecture implies that the right vertical map is
an isomorphism on $\pi_{q}$ for $q>1$.  For $q=1$,
$\pi_{1}(K(\bZ)\phat;\bQ)=0$ and for $q=0$,
$\pi_{0}(K(\bZ)\phat;\bQ)\iso\bQ\phat$ generated by the unit of the ring
structure, and so the map is an injection on these groups as
well. This establishes the claim for the right vertical map,
which then establishes the claim for the lefthand bottom
horizontal map.  For the last map, we look at the $K(1)$-localization
of the Quillen localization sequence for~$\bQ$,
\[
L_{K(1)}\biggl(\bigvee_{\ell\text{ prime}} K(\bF_{\ell})\biggr)\to 
L_{K(1)}K(\bZ)\to L_{K(1)}K(\bQ)\to \Sigma \dotsb.
\]
Except in degrees $0$ and $-1$, the homotopy groups of each
$L_{K(1)}K(\bF_{\ell})$ are finite and in odd degrees; in degrees $0$
and $-1$, the groups are $\bZ\phat$.  For
$q>0$, we have
$\pi_{q}(L_{K(1)}K(\bZ);\bQ)=0$ for $q\equiv 0,2,3\pmod 4$,
while $\pi_{q}(L_{K(1)}K(\bZ))\iso \bZ\phat$ or $\bZ\phat[2]\oplus
\bZ/2$ for $q\equiv 1\pmod 4$, and $\pi_{0}(L_{K(1)}K(\bZ))\iso
\bZ\phat$.  Since the $K(1)$-localization of the
wedge is the $p$-completion of the wedge of the $K(1)$-localizations,
the map
\[
L_{K(1)}\biggl(\bigvee_{\ell\text{ prime}} K(\bF_{\ell})\biggr)\to 
L_{K(1)}K(\bZ)
\]
must be zero on $\pi_{q}(-;\bQ)$ for $q\geq 1$, and by inspection, it
is also zero for $q=0$ since each map 
\[
\pi_{0}(L_{K(1)}K(\bF_{\ell}))\to
\pi_{0}(L_{K(1)}K(\bZ))
\]
is the zero map and the contribution of $\pi_{-1}(-)$ to
$\pi_{0}((-)\phat)$ in this case is 
\[
\Hom(\bZ/p^{\infty},\bigoplus \bZ\phat)=0.
\]
This
proves the claim for the map $L_{K(1)}K(\bZ)\to L_{K(1)}K(\bQ)$ and
completes the proof of the lemma in the case $S=\bS$. 

In the case $S=\bS_{(p)}$, we have that the map
$L_{K(1)}K(\bZ_{(p)})\to L_{K(1)}K(\bQ)$ is a weak equivalence (since
the fiber is $L_{K(1)}K(\bF_{p})\simeq *$).  Thus, it suffices to see
that the map 
\begin{equation}\label{eq:lasteq}
\holim L_{n}K(\bS_{(p)})\phat\to L_{K(1)}K(\bZ_{(p)})
\end{equation}
is injective on rational homotopy groups in non-negative degrees.  
For the purposes of this argument, denote the cofiber of the
map~\eqref{eq:lasteq} as $\Cof$.  By 
Theorem~\ref{thm:L1DGM} and Proposition~\ref{prop:fixTCS}, $\Cof$ is
weakly equivalent to the cofiber of the map 
\[
TC(\bS_{(p)})\phat\to L_{K(1)}TC(\bZ_{(p)}).
\]
In the diagram
\[
\xymatrix{%
TC(\bS)\phat\ar[r]\ar[d]&L_{K(1)}TC(\bZ)\ar[d]\\
TC(\bS_{(p)})\phat\ar[r]&L_{K(1)}TC(\bZ_{(p)})
}
\]
the vertical maps are weak equivalences (for example,
by~\cite[2.5]{BMY1}), and so the horizontal maps have equivalent
cofibers.  Applying Theorem~\ref{thm:L1DGM} and
Proposition~\ref{prop:fixTCS} again, we see that the cofiber on top is
weakly equivalent to the cofiber of the map
\begin{equation}\label{eq:reallylasteq}
\holim L_{n}K(\bS)\phat\to L_{K(1)}K(\bZ).
\end{equation}
Work of the previous paragraph now shows that this cofiber (and
therefore also $\Cof$) satisfies
$\pi_{q}(-;\bQ)=0$ for $q>1$.   Thus, the map~\eqref{eq:lasteq} is
injective on $\pi_{q}(-;\bQ)$ for $q\geq 1$.  

It remains to check that the map~\eqref{eq:lasteq} is
injective on $\pi_{0}(-;\bQ)$.  In the work on the case $S=\bS$, we
showed that the map~\eqref{eq:reallylasteq} 
is injective on $\pi_{0}(-;\bQ)$ and it follows from the work in the
previous paragraph that the map $L_{K(1)}K(\bZ)\to \Cof$
is surjective on $\pi_{1}(-;\bQ)$.  Since this map factors through
the map $L_{K(1)}K(\bZ_{(p)})\to \Cof$, that map is also surjective on
$\pi_{1}(-;\bQ)$. We conclude that the map~\eqref{eq:lasteq} is
injective on $\pi_{0}(-;\bQ)$, which completes the proof of the lemma.
\end{proof}

%%%%%%%%%%%%%%%%%%%%%%%%%%%%%%%%%%%%%%%%%%%%%%%%%%%%%%%%%%%%%%%%%%%%%%%%
% Bibliography
%%%%%%%%%%%%%%%%%%%%%%%%%%%%%%%%%%%%%%%%%%%%%%%%%%%%%%%%%%%%%%%%%%%%%%%%

\bibliographystyle{plain}
\bibliography{bluman}

\end{document}